\documentclass[11pt,oneside]{elsarticle}
\pdfoutput=1
  \pdfpageattr{/Group <</S /Transparency /I true /CS /DeviceRGB>>}

\makeatletter
\def\ps@pprintTitle{%
 \let\@oddhead\@empty
 \let\@evenhead\@empty
 \def\@oddfoot{}%
 \let\@evenfoot\@oddfoot}
\makeatother

\usepackage{etex}

\usepackage[T1]{fontenc}
\usepackage[english]{babel}
\usepackage{hyperref}
\usepackage{geometry}
\usepackage{pdflscape}
\usepackage{amsfonts,amsmath,amssymb,amsthm}
\usepackage{tikz}
\usepackage{pgfplots}
\usepackage{caption}
\usepackage[labelfont=rm]{subcaption}
\usepackage{xspace}
\usepackage{booktabs,tabularx}
\usepackage{dcolumn}
\usepackage{listings}
\usepackage{marvosym}
\usepackage{microtype}
\usepackage{enumitem}
\usepackage{blkarray}

\microtypesetup{tracking,kerning,spacing}
\microtypecontext{spacing=nonfrench}

\newcommand{\inclpic}[1]{\includegraphics{#1}}

\definecolor{links}{rgb}{.2,.1,.5}
\definecolor{cites}{rgb}{.5,.1,.2}
\hypersetup{colorlinks,linkcolor=links,citecolor=cites}

\newcolumntype{d}[1]{D{.}{.}{#1} }

\theoremstyle{plain}
\newtheorem{theorem}{Theorem}
\newtheorem{lemma}[theorem]{Lemma}
\newtheorem{proposition}[theorem]{Proposition}
\newtheorem{corollary}[theorem]{Corollary}
\theoremstyle{definition}
\newtheorem{remark}[theorem]{Remark}
\newtheorem{example}[theorem]{Example}

\newtheorem{problem}[theorem]{Problem}

\newcommand\cA{{\mathcal A}}
\newcommand\cB{{\mathcal B}}
\newcommand\cG{{\mathcal G}}
\newcommand\cH{{\mathcal H}}
\newcommand\cK{{\mathcal K}}
\newcommand\cM{{\mathcal M}}

\newcommand\RR{{\mathbb R}}
\newcommand\TT{{\mathbb T}}

\newcommand{\1}{\mathbf 1}
\newcommand{\0}{\mathbf 0}

\newcommand\SetOf[2]{\left\{\left.#1\vphantom{#2}\ \right|\ #2\vphantom{#1}\right\}}
\newcommand\smallSetOf[2]{\{{#1}\,|\,{#2}\}}
\DeclareMathOperator{\conv}{conv}
\DeclareMathOperator{\pos}{pos}
\DeclareMathOperator{\supp}{supp}

\newcommand\pair[2]{({#1},{#2})}
\newcommand\triple[3]{({#1},{#2},{#3})}
\newcommand\dcup{\sqcup}  
\newcommand\transpose[1]{{#1}^\top}

\newcommand\TPmin{\mathbb{TP}_{\min}}

\DeclareMathOperator{\tdet}{tdet}
\DeclareMathOperator{\tcone}{tcone}
\DeclareMathOperator{\Sym}{Sym}
\DeclareMathOperator{\tangentgraph}{TG}

\newcommand\envelope{{\mathcal E}} 
\DeclareMathOperator{\Ord}{Ord} 
\newcommand\cF{{\mathcal F}}
\newcommand\Bgraph{{\rm B}} 
\newcommand\fan{\Delta}
\newcommand\typedecomp[1]{{\mathcal T}(#1)}
\newcommand\typegraph[1]{{\rm T}(#1)}
\newcommand\thalf{\operatorname{thalf}}
\newcommand\olthalf{\overline{\operatorname{thalf}}}
\newcommand\signedcell[2]{#1_{#2}} 
\newcommand\csec[2]{\overline{S_{#1}(#2)}}

\title{Weighted digraphs and tropical cones}

\author[tub]{Michael Joswig\corref{cor1}\fnref{fn1}}
\ead{joswig@math.tu-berlin.de}

\author[tub]{Georg Loho\corref{cor1}}
\ead{loho@math.tu-berlin.de}

\address[tub]{Institut f{\"u}r Mathematik, MA 6-2, TU Berlin, Str.\ des 17. Juni 136, 10623 Berlin, Germany}

\cortext[cor1]{Corresponding author}

\fntext[fn1]{M.~Joswig is partially supported by Einstein Foundation Berlin and Deutsche Forschungsgemeinschaft (DFG) within the Priority Program 1489 ``Experimental Methods in Algebra, Geometry and Number Theory''.}

\begin{document}
\begin{abstract}
  This paper is about the combinatorics of finite point configurations in the tropical projective space or, dually, of
  arrangements of finitely many tropical hyperplanes.  Moreover, arrangements of finitely many tropical halfspaces can
  be considered via coarsenings of the resulting polyhedral decompositions of $\RR^d$.  This leads to natural cell
  decompositions of the tropical projective space $\TPmin^{d-1}$.  Our method is to employ a known class of ordinary
  convex polyhedra naturally associated with weighted digraphs.  This way we can relate to and use results from
  combinatorics and optimization.  One outcome is the solution of a conjecture of Develin and Yu (2007).
\end{abstract}

\maketitle

\section{Introduction}
\noindent
The tradition of max-plus linear algebra in optimization and related areas goes back several decades; for an overview,
e.g., see Litvinov, Maslov and Shpiz \cite{LitvinovMaslovShpiz:2001}, Cohen, Gaubert and Quadrat
\cite{CohenGaubertQuadrat:2004} or Butkovi\v{c} \cite{Butkovic:10} and their references.  Develin and Sturmfels connected
max-plus linear algebra under the name of \emph{tropical convexity} to geometric combinatorics in their landmark
paper~\cite{DevelinSturmfels:2004}; see also \cite[Chapter~5]{TropicalBook}.  This line of research has been continued
in \cite{MJ:2005}, \cite{DevelinYu:2007}, \cite{ArdilaDevelin:2009}, \cite{FinkRincon:1305.6329} and other references.
The interest in a more geometric perspective comes from several directions.  One source is tropical geometry, which,
e.g., relates tropical convexity to the combinatorics of the Grassmannians \cite{SpeyerSturmfels04}, 
\cite{HerrmannJoswigSpeyer14}, \cite{FinkRincon:1305.6329}.  A second independent source is the study of tropical
analogues of linear programming \cite{ABGJ-Simplex:A} which, e.g., is motivated by its connections to deep open problems
in computational complexity \cite{AkianGaubertGutermann12}.

Since the paper \cite{DevelinSturmfels:2004} by Develin and Sturmfels more than ten years ago some of the strands of
research still seem to diverge.  The main purpose of this paper is to help bridging this gap.  Our point of departure is
\cite[Theorem~1]{DevelinSturmfels:2004}, which establishes a fundamental correspondence between the configurations of $n$
points in the \emph{tropical projective torus} $\RR^d/\RR\1$ and the regular subdivisions of the product of simplices
$\Delta_{d-1}\times\Delta_{n-1}$.  We suggest to call this result the \emph{Structure Theorem of Tropical Convexity}.
It was recently extended by Fink and Rinc\'on \cite[Corollary~4.2]{FinkRincon:1305.6329} to include regular subdivisions
of subpolytopes of products of simplices.  For the tropical point configurations this amounts to taking $\infty$ as a
coordinate into account.  Our first contribution is a new proof of that result (Corollary~\ref{coro:dual_sub_tcone}).
Moreover, in \cite{DevelinSturmfels:2004} and \cite{FinkRincon:1305.6329} only tropical convex hulls of points (or
dually, arrangements of tropical hyperplanes) are considered, whereas here we also bring exterior descriptions in terms
of tropical half-spaces \cite{MJ:2005}, \cite{GaubertKatz:11} into the picture.  Arrangements of max-tropical halfspaces
correspond to the `two-sided max-linear systems' in the max-plus literature \cite[\S7]{Butkovic:10}.  As an additional
benefit our methods allow us to resolve a previously open question raised by Develin and Yu, who conjectured that a
finitely generated tropical convex hull is pure and full-dimensional if and only if it has a half-space description in
which the apices of these tropical half-spaces are in general position \cite[Conjecture~2.11]{DevelinYu:2007}.  We show
that, indeed, general position implies pureness and full-dimensionality (Theorem~\ref{thm:pure}), and we give a
counter-example to the converse (Example~\ref{exmp:pure}).  The approach through tropical convex hulls on the one hand
and the approach through systems of tropical inequalities on the other hand gives rise to two interesting cell
decompositions of the \emph{tropical projective spaces} (Theorem~\ref{thm:covector_decomposition} and
Corollary~\ref{cor:signed_cell_decomposition}).  This ties in with compactifications of tropical varieties; see
Mikhalkin~\cite[\S3.4]{Mikhalkin:2006}.

As in \cite{DevelinSturmfels:2004} it turns out to be convenient to examine the regular subdivisions of products of
simplices and their subpolytopes in terms of a dual ordinary convex polyhedron, which we call the \emph{envelope} of the
tropical point configuration.  In fact, it is even fruitful to see this envelope as a special case of a more general
class of ordinary polyhedra which are associated with directed graphs with weighted arcs.  These \emph{weighted digraph
  polyhedra} are defined by linear inequalities of the form
\[
x_i - x_j \ \leq \ w_{ij} \enspace ,
\]
where $w_{ij}$ is the weight on the arc from the node $i$ to the node $j$.  Their feasible points are well known as
\emph{potentials} in the optimization literature, and the weighted digraph polyhedra are sometimes called `shortest path
polyhedra'; e.g., see \cite[\S8.2]{Schrijver:CO:A} for an overview.  Recently potentials and weighted digraph
polyhedra starred prominently in the work of Khachiyan and al.~\cite{Khachiyan:2008} on hardness results in the context of
vertex enumeration.
Specializing all arc weights to zero yields the \emph{braid cones} of Postnikov, Reiner and Williams \cite{PostnikovReinerWilliams08}, which are closely related to \emph{order polytopes} of partially ordered sets.
By applying a celebrated result of Stanley \cite[Theorem~1.2]{Stanley:1986} we obtain a combinatorial characterization of the entire face lattice of any digraph
cone (Theorem~\ref{thm:partition}).

Our paper is organized as follows.  Section~\ref{sec:weighted} starts out with investigating a general weighted digraph
polyhedron $Q(W)$ associated with a $k{\times}k$-matrix $W$, which we read as a directed graph $\Gamma=\Gamma(W)$
equipped with a weight function.  The braid cones, with all finite entries equal to zero, naturally come in as their
recession cones.  We show that the face lattice of a braid cone is isomorphic to a face figure of the order polytope
associated with the acyclic reduction of $\Gamma$ and, via Stanley's result \cite[Theorem~1.2]{Stanley:1986}, to a
partially ordered set of partitions of the node set of $\Gamma$ ordered by refinement. It is a key observation that the
faces of a weighted digraph polyhedron are again weighted digraph polyhedra.  The envelope of an arbitrary
$d{\times}n$-matrix $V$ is the weighted digraph polyhedron for a specific $(d{+}n){\times}(d{+}n)$-matrix constructed
from~$V$.

In Section~\ref{sec:tropical} we direct our attention to tropical convexity, which is essentially the same as linear
algebra over the tropical semi-ring $\TT_{\min}=(\RR\cup\{\infty\},\min,+)$.  Clearly, it is just a matter of taste if
one prefers $\min$ or $\max$ as the tropical addition.  More importantly though, it turns out to be occasionally
convenient to use both these operations together to be able to phrase some of our results in a natural way.  So we
usually consider tropical linear spans of vectors in the $\min$-tropical setting and intersections of 
tropical half-spaces in the $\max$-setting.  With any matrix $V\in\RR^{d\times n}$ Develin and Sturmfels associate a
polyhedral decomposition of the tropical projective torus $\RR^d/\RR\1$ \cite[\S3]{DevelinSturmfels:2004}; here $\1$
denotes the all ones vector.  We follow Fink and Rinc\'on \cite{FinkRincon:1305.6329} in calling this polyhedral complex
the \emph{covector decomposition}.  The cells of the covector decomposition are naturally indexed by subgraphs of the
digraph $\Gamma(W)$, where $W$ is the $(d{+}n){\times}(d{+}n)$-matrix mentioned above.  Moreover, these cells arise as
orthogonal projections of the faces of the envelope of $V$.  If $V$ is finite then (in the tropical projective torus)
the union of the bounded cells of the type decomposition is exactly the tropical convex hull of the columns of $V$.
Further, the covector decomposition is dual to a regular subdivision of the product of simplices
$\Delta_{d-1}\times\Delta_{n-1}$.  If $V$ has infinite coordinates, it still makes sense to talk about the
\emph{tropical cone} generated by the columns, but $\Delta_{d-1}\times\Delta_{n-1}$ gets replaced by the subpolytope
corresponding to the finite entries of~$V$; see \cite{FinkRincon:1305.6329}. This leads to studying point configurations
in the \emph{tropical projective space}; see Mikhalkin~\cite[\S3.4]{Mikhalkin:2006} and Section~\ref{sec:projective}
below.  Another way of interpreting the matrix $V$, with coefficients in $\TT_{\min}$, is as an arrangement of
max-tropical hyperplanes.  The covector decomposition arises as the common refinement of the affine fans corresponding
to these tropical hyperplanes.  Equipping such a tropical hyperplane arrangement with a certain graph encoding the
feasibility of a cell gives rise to a max-tropical cone described as the intersection of finitely many \emph{tropical
  half-spaces}; see \cite{MJ:2005} and \cite{GaubertKatz:11}.  This is how tropical cones naturally arise in the context
of tropical linear programming.  In \cite{ABGJ-Simplex:A} a tropical version of the simplex method is described.  The
pivoting operation proposed there can be explained in terms of operations on the graph $\Gamma(W)$, the crucial object
being the \emph{tangent digraph} from \cite[\S3.1]{ABGJ-Simplex:A}, which carries the same information as the `tangent
hypergraphs' of Allamigeon, Gaubert and Goubault \cite{AGG:2013}.  We show how the tangent digraph encodes the local
combinatorics of the covector decomposition induced by $V$ in the neighborhood of a given point. Finally, we recall the
\emph{signed cell decompositions} from \cite[\S3.2]{ABGJ-Simplex:A} which form the tropical analogues of the polyhedral
complexes generated from a system of ordinary affine hyperplanes.

The upshot is that all the remarkable combinatorial properties of tropical convexity can be inferred from the weighted
digraph polyhedra.  It is worth noting that the facet normals of their defining inequalities are precisely the roots of
a type A root system. Lam and Postnikov \cite{LamPostnikov:2007} introduced `alcoved polytopes' which are exactly the
weighted digraph polyhedra which are bounded (modulo projecting out the subspace $\RR\1$).  These are also the
\emph{polytropes} in \cite{JoswigKulas:2010}. Section~\ref{sec:polytropes} gives more details.  The paper closes with a
few open problems.

\section{Weighted digraph polyhedra} \label{sec:weighted}
\subsection{The construction}
Let $W=(w_{ij})$ be an arbitrary $k{\times}k$-matrix with coefficients in $\TT_{\min}=\RR\cup\{\infty\}$.  This yields a digraph
$\Gamma(W)$ with node set $[k]$ and an arc from $i$ to $j$ whenever the coefficient $w_{ij}$ is finite.  Notice that
$\Gamma(W)$ may have loops, corresponding to finite entries on the diagonal.  Also $(i,j)$ and $(j,i)$ both may be arcs,
but there are no other multiple edges.  The matrix $W$ induces a map, $\gamma$, which assigns to each arc $(i,j)$ of
$\Gamma(W)$ its \emph{weight} $w_{ij}$.  We call the pair $(\Gamma(W),\gamma(W))$ the \emph{weighted digraph} associated
with $W$. Conversely, each finite directed graph $\Gamma$ endowed with a weight function $\gamma$ on its arcs has a
\emph{weighted adjacency matrix} $W(\Gamma,\gamma)$.  Often we will not distinguish between the matrix $W$ and the
digraph $\Gamma$ equipped with the weight function $\gamma$.

Our key player is the \emph{weighted digraph polyhedron} $Q(W)$ in $\RR^k$ which is defined by the linear inequalities
\begin{equation}\label{eq:defining}
  x_i - x_j \ \leq \ w_{ij} \qquad \text{ for each arc $(i,j)$ in $\Gamma(W)$} \enspace .
\end{equation}
For a directed graph $\Gamma$ with a weight function $\gamma$ we also write $Q(\Gamma,\gamma)$ instead of
$Q(W(\Gamma,\gamma))$.  Observe that $-Q(W) = Q(\transpose{W})$.  A feasible point in $Q(W)$ is sometimes called a \emph{potential} on the digraph $\Gamma$; e.g.,
see \cite[\S8.2]{Schrijver:CO:A}.  The following result of Gallai \cite{Gallai:58} clarifies the feasibility of the constraints; see
also \cite[Theorem~8.2]{Schrijver:CO:A} and \cite[\S2.1]{Butkovic:10}. 

\begin{lemma} \label{lemma:feasible}
  The weighted digraph polyhedron $Q(W)$ is empty if and only if the weighted digraph $(\Gamma,\gamma)$ has a negative cycle.
\end{lemma}

If the weighted digraph $(\Gamma,\gamma)$ does not have any negative cycle there is a directed shortest path between any
two nodes.  Let $W^*=(w_{ij}^*)$ be the $k{\times}k$-matrix which records the weights of these shortest paths.
Following Butkovi\v{c} \cite[\S1.6.2]{Butkovic:10} we call the shortest path matrix $W^*$ the \emph{Kleene star} of $W$.
The tropical addition $\oplus=\min$ extends to vectors and matrices coefficientwise.  Moreover, the tropical addition
and the tropical multiplication give rise to a tropical matrix multiplication, which we also write as $\odot$.  Matrix
powers of $W$ with respect to $\odot$ are written as $W^{\odot \ell}$ where $W^{\odot 0}=I$ is the min-tropical unit matrix,
which has zero coefficients on the diagonal and $\infty$ otherwise, and $W^{\odot (\ell+1)}=W^{\odot \ell}\odot W$.
With this notation we have the formula
\[
W^* \ = \ I \oplus W \oplus W^{\odot 2} \oplus \cdots \oplus W^{\odot k} \enspace ,
\]
whose direct evaluation amounts to applying the Bellman-Ford method for computing all shortest paths \cite[\S8.3]{Schrijver:CO:A}.
The next lemma points out a special property of the inequality description given by $W^*$; see \cite[Theorem~8.3]{Schrijver:CO:A}.
\begin{lemma} \label{lem:Kleene_tight}
  Each of the defining inequalities from \eqref{eq:defining} for the weighted digraph polyhedron of the matrix $W^*$ is tight. 
\end{lemma}
\begin{proof}
  Let $x_i - x_j \leq w_{ij}^*$ be an inequality defining $Q(W^*)$.  The vector of weights $w_{pj}^*$ for $p \in [k]$,
  i.e., the $j$th column of $W^*$, satisfies each inequality by the shortest path property $w_{pj}^* \leq w_{pq}^* +
  w_{qj}^*$.  Equivalently we have $w_{pj}^* - w_{qj}^* \leq w_{pq}^*$. Due to $w_{jj}^* = 0$, this vector satisfies the
  equality $x_i - x_j = w_{ij}^*$.
\end{proof}

Throughout the following we assume that $(\Gamma,\gamma)$ does not have a negative cycle.  In view of Lemma~\ref{lemma:feasible}
this is equivalent to the feasibility of $Q(W)$, and the Kleene star $W^*$ is defined.  Further, let $E(W)$ be the
\emph{equality graph} of $W$, which is the undirected graph on the node set $[k]$ and which has an edge between $i$ and
$j$ if $Q(W)$ satisfies $x_i-x_j=w_{ij}^*<\infty$ or $x_j-x_i=w_{ji}^*<\infty$.
\goodbreak 

\begin{lemma} \strut \label{lemma:zero_cycles}
  \begin{enumerate} [label = (\alph*)]
  \item We have $Q(W^*) = Q(W)$ and $E(W^*)=E(W)$.
  \item \label{lemma:item:zero_cycles} Two distinct nodes $i$ and $j$ are contained in a directed cycle of weight zero in
    $\Gamma(W)$ if and only if $\{i,j\}$ is contained in the equality graph $E(W)$ if and only if
    $w_{ij}^*=-w_{ji}^*<\infty$.
  \end{enumerate}
\end{lemma}
\begin{proof}
  The proof for both statements is essentially the same.
  Let $\pi=(i_0,i_1,\dots,i_m)$ be a directed path in $\Gamma$.  This corresponds to the inequalities $x_{i_{\ell-1}} \leq x_{i_{\ell}} + w_{i_{\ell-1}i_{\ell}}$ for $\ell \in \{1, \ldots, m\}$. By transitivity we obtain
  \[
  x_{i_0} \ \leq \ x_{i_m} + \sum_{\ell = 1}^{m}w_{i_{\ell-1}i_{\ell}}
  \]
  as a valid inequality for $Q(W)$.  Restricting to shortest paths shows $Q(W^*)\supseteq Q(W)$.  The other inclusion is obvious.  Notice that this readily implies that the equality graphs $E(W)$ and $E(W^*)$ are the same.
  
  Now suppose that $\pi$ is a directed cycle of weight zero.  In particular, $i_0=i_m$ is the same node and because of the presumed feasibility, the cycle contains the shortest path for any pair of its nodes.  The above yields for each $\mu \in \{0, \ldots, m\}$ the inequalities
  \[
  x_{i_0} \ \leq \ x_{i_{\mu}} + \sum_{\ell = 1}^{\mu}w_{i_{\ell-1}i_{\ell}} \ = \ x_{i_{\mu}} + w_{i_0,i_{\mu}}^* \text{ and }  x_{i_{\mu}} \ \leq \ x_{i_m} + \sum_{\ell = \mu+1}^{m}w_{i_{\ell-1}i_{\ell}} \ = \ x_{i_0} + w_{i_{\mu},i_{0}}^*\enspace .
  \]
With $w_{i_0,i_{\mu}}^* + w_{i_{\mu},i_0}^* = 0$ we obtain 
\[
x_{i_0} - x_{i_{\mu}} \ \leq \ w_{i_0,i_{\mu}}^* \ = \ -w_{i_{\mu},i_0}^* \ \leq \ x_{i_0} - x_{i_{\mu}}
\]
  and hence the equality $x_{i_0}-x_{i_{\mu}} = w_{i_0,i_{\mu}}^*$.  This shows that the edge $\{i_0,i_{\mu}\}$ is contained in the equality graph $E(W^*)=E(W)$.

  Finally, let $\{i,j\}$ be an edge in $E(W)=E(W^*)$.  Then $x_{i}-x_{j} = w_{ij}^*<\infty$, and it follows that also
  $x_{j}-x_{i}=-w_{ij}^*$ is finite.  Since the inequality $x_j - x_i \leq w_{ji}^*$ is tight by Lemma~\ref{lem:Kleene_tight} we obtain $w_{ji}^* = -w_{ij}^*$.
  Therefore, there is a directed path from $j$ to $i$ in $\Gamma(W)$, and hence $(i,j,i)$ is a directed cycle of weight zero in $\Gamma(W^*)$.
  From this we infer our claim.
\end{proof}

\begin{figure}[ht]
  \centering
  \resizebox{0.6\textwidth}{!}{\inclpic{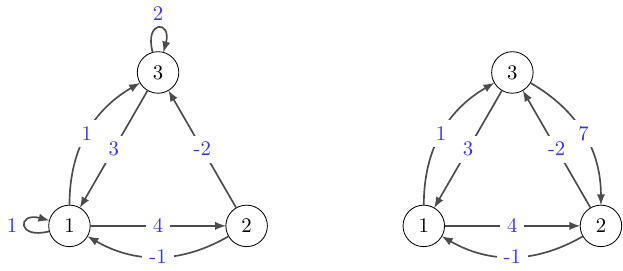}}
  \caption{The directed graphs defined by the matrices $W$ and $W^*$ from Example~\ref{example:weighted_digraph_W}}
  \label{figure:weighted_digraph_W}
\end{figure}

\begin{example}\label{example:weighted_digraph_W}
  The $3{\times}3$ matrix
  \begin{equation}\label{eq:W33}
    W \ = \
    \left(
      \begin{array}[c]{ccc}
        1 & 4 & 1 \\
        -1 & 0 & -2 \\
        3 & \infty & 2
      \end{array}
    \right)
  \end{equation}
  defines a directed graph without any cycles of weight zero. Its Kleene star is the matrix
  \[
  W^* \ = \
  \left(
    \begin{array}[c]{ccc}
      0 & 4 & 1 \\
      -1 & 0 & -2 \\
      3 & 7 & 0
    \end{array}
  \right) \enspace .
  \]
  The graphs of $W$ and $W^*$ are displayed in Figure~\ref{figure:weighted_digraph_W}, while
  Figure~\ref{figure:polyhedron_W} shows the corresponding weighted digraph polyhedron.  Our convention for drawing
  digraphs is to omit loops of weight zero and arbitrary arcs of infinite weight.  Since each weighted digraph
  polyhedron contains the one-dimensional linear subspace $\RR\1$ in its lineality space, throughout we draw pictures in
  the quotient $\RR^d/\RR\1$, which is called the \emph{tropical projective $(d{-}1)$-torus} in
  \cite[\S5.2]{TropicalBook}.  More precisely, for a feasible point $x+\RR\1$ in the quotient we draw the unique
  representative with $x_1=0$.  This is the same as drawing the intersection of $Q(W)$ with the hyperplane $x_1=0$.
  As the polyhedron $Q(W)$ corresponding to the matrix~\eqref{eq:W33} is not contained in any hyperplane its equality graph $E(W)$ is the undirected graph with three isolated nodes.
\end{example}

\begin{figure}[ht]
  \centering
  \resizebox{0.4\textwidth}{!}{\inclpic{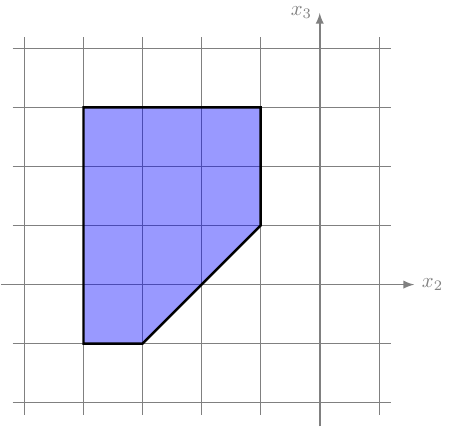}}
  \caption{The weighted digraph polyhedron $Q(W)=Q(W^*)$ for the matrices $W$ and $W^*$ from
    Example~\ref{example:weighted_digraph_W}, shown in the tropical projective $2$-torus}  
  \label{figure:polyhedron_W}
\end{figure}

We return to studying general matrices $W$.
\begin{lemma}\label{lemma:dim equality graph}
  The connected components of the equality graph of $E(W)$ are complete graphs, and their number is the dimension of the
  polyhedron $Q(W)$.
\end{lemma}
\begin{proof}
  The equalities $x_i - x_j = w_{ij}^*$ and $x_j - x_{\ell} = w_{j \ell}^*$ imply $x_i - x_{\ell} = w_{ij}^* + w_{j \ell}^* \geq
  w_{i\ell}^*$ and therefore $x_i - x_{\ell} = w_{i\ell}^*$ for any three nodes $i,j,\ell$ in the equality graph.  So there is an edge
  between any two nodes in a connected component of $E(W)$.  The statement about the dimension follows as the equality
  graph summarizes exactly those inequalities which are attained with equality and the connected components form a
  partition of the node set.
\end{proof}
The lemma above says that the equality graph encodes an equivalence relation on the node set $[k]$. The partition into
the connected components is the \emph{equality partition}.  Abusing our notation, again we denote this partition as $E(W)$.

\subsection{Intersections and faces}
Throughout the following we will frequently consider several graphs which share the same set of nodes.  In this case it
makes sense to identify such a graph with its set of edges (or arcs, in the directed case).  This allows to talk about
intersections and unions of such graphs.
\begin{lemma} \label{lemma: intersection wgp} Let $U$ and $W$ be $k{\times}k$-matrices.  The intersection of the
  weighted digraph polyhedra $Q(U)$ and $Q(W)$ is the weighted digraph polyhedron $Q(U \oplus W)$.  The arc set of the
  graph $\Gamma(U \oplus W)$ is the union of $\Gamma(U)$ and $\Gamma(W)$.
\end{lemma}
\begin{proof}
  The intersection of two polyhedra is given by the union of their defining inequalities. The two inequalities of the
  form $x_i - x_j \leq u_{ij}$ and $x_i - x_j \leq w_{ij}$ are both satisfied if and only if the inequality $x_i - x_j
  \leq \min(u_{ij},w_{ij})$ holds.
\end{proof}

Again we assume that the graph $\Gamma(W)$ does not contain any negative cycle, and thus $Q(W)$ is feasible.  Each face of
the polyhedron $Q(W)$ is obtained by turning some of the defining inequalities into equalities.  More precisely, for any
subgraph $G$ of $\Gamma$ let
\[
F_G \ = \ F_G(W)\ = \ F_G(\Gamma,\gamma)\ = \ \SetOf{x\in Q(W)}{x_i-x_j=w_{ij} \text{ for all } (i,j)\in G} \enspace .
\]
By construction $F_G$ is a face of $Q(W)$, and conversely each face of $Q(W)$ arises in this way.  We define a new
$k{\times}k$-matrix, denoted $W\#G$; it is constructed from $W$ by replacing the entries $w_{ji}$ with $-w_{ij}$ for each
$(i,j)\in G$. If $G$ contains both $(i,j)$ and $(j,i)$ as arcs, this operation is only defined provided that $w_{ij} +
w_{ji} = 0$.
The reason is that this equality is implied by $x_i-x_j=w_{ij}$ combined with $x_j-x_i=w_{ji}$.
The following is immediate.
\begin{lemma} \label{lemma: face of wdp}
  Faces of weighted digraph polyhedra are weighted digraph polyhedra.  More precisely, 
  \[
  \begin{aligned}
    F_G(W) \ &= \ Q(W) \cap \SetOf{x \in \RR^k}{x_i - x_j=w_{ij} \text{ for $(i,j)\in G$}} \\
    &= \ Q(W) \cap \SetOf{x \in \RR^k}{x_j - x_i \leq -w_{ij} \text{ for $(i,j)\in G$}} \ = \ Q(W\#G) \enspace .
  \end{aligned}
  \]
  Furthermore, the equality partition $E(W\#G)$ of a face $F_G(W)$ is obtained from the equality partition $E(W)$ by
  uniting the two parts which contain $i$ and $j$ if $(i,j)$ is an arc in $G$.
\end{lemma}
By Lemma~\ref{lemma:dim equality graph} the
dimension of the face $F_G(W)$ equals the size of the partition $E(W\#G)$.

\begin{example}\label{example:face}
  If $W$ is the matrix from Example \ref{example:weighted_digraph_W} and $G$ consists of the single arc $(2,3)$ then 
  we have
  \[
  W\#G \ = \
  \begin{pmatrix}
    1 & 4 & 1 \\
    -1 & 0 & -2 \\
    3 & 2 & 2
  \end{pmatrix} \enspace .
  \]
The equality graph $E(W\#G)$ consists of the isolated node $1$, and the nodes $2$ and $3$ are joined by an edge. This reflects that $Q(W\#G)$ is contained in the supporting hyperplane induced by the equality from $G$. Finally, the equality partition is $\{\{1\},\{2,3\}\}$.
\end{example}

\subsection{Braid Cones}
We will now apply our previous results to the situation where the weight function is constantly zero on the arcs.
Then for an arbitrary digraph $\Gamma$ the weighted digraph polyhedron
\[
Q(\Gamma,\0) \ = \ \SetOf{ x \in \RR^k }{ x_i\ \leq\ x_j \text{ for all } (i,j) \in \Gamma }
\]
is a polyhedral cone, the \emph{braid cone} of $\Gamma$ studied by Postnikov, Reiner and Williams \cite{PostnikovReinerWilliams08}.
See, in particular, \cite[\S3.4]{PostnikovReinerWilliams08} for detailed information about their combinatorial structure.
Here we wish to relate braid cones to order polytopes.

All points in the subspace $\RR\1$ are feasible.
Since every cycle has weight zero, applying Lemma \ref{lemma:zero_cycles}\ref{lemma:item:zero_cycles} to the cone $Q(\Gamma,\0)$ yields the
following.
\begin{proposition} \label{prop:cones_components}
  The parts of the equality partition $E(W(\Gamma,\0))$ are exactly the strong components of $\Gamma$. In particular,
  the dimension of the braid cone $Q(\Gamma,\0)$ equals the number of strong components of $\Gamma$.
\end{proposition}

Any hyperplane of the form $x_i=x_j$ defines a \emph{split} of the unit cube $[0,1]^k$, i.e., it defines a (regular) subdivision of the unit cube into two subpolytopes; see \cite{HerrmannJoswig:2008}.
Notice that such a split hyperplane does not separate any edge of the unit cube.
Let us look at the map $\kappa$ which sends each face $F$ of the braid cone $Q(\Gamma,\0)$ to the intersection $F \cap [0,1]^k$.
Clearly, this intersection is never empty (unless $F$ is).

Now suppose that $\Gamma$ is acyclic. Then those inequalities which define facets of $Q(\Gamma,\0)$ correspond to the
covering relations of the partially ordered set $P(\Gamma)$ on the node set $[k]$ of $\Gamma$ induced by the arcs.  It
follows that $\kappa(Q(\Gamma,\0))=Q(\Gamma,\0)\cap[0,1]^k$ is the \emph{order polytope} $\Ord(\Gamma)$ of the poset $P(\Gamma)$.  The poset
$P(\Gamma)$ describes the transitive closure of the relation defined on the set $[k]$ by the arcs of $\Gamma$.
Conversely, each finite poset gives rise to a directed graph whose nodes are the elements and the arcs are given by the
covering relations directed, say, upwards.

The order polytope $\Ord(\Gamma)$ contains the points $\0$ and $\1$ as vertices.  Therefore there exists a unique
minimal face which contains both of them; denote this face by $F_{01}$.  Note that the dimension of $F_{01}$ can be any
number between $1$ (if $F_{01}$ is the edge $[\0,\1]$) and $k$ (if the graph $\Gamma$ does not contain any edges).  The
\emph{face figure} of $F_{01}$, written as $\cF_{01}$, is the principal filter of the element $F_{01}$ in the face poset
of the order polytope $\Ord(\Gamma)$.  The subposet $\cF_{01}$ is the face poset of a polytope of dimension $k-\dim
F_{01}-1$.  The face figure $\cF_{01}$ consists of exactly those faces of $\Ord(\Gamma)$ which are not contained in any
facet of the cube $[0,1]^k$.  It is immediate that $\kappa$ maps faces of the braid cone $Q(\Gamma,\0)$ to the faces
of the order polytope $\Ord(\Gamma)$ which lie in the face figure $\cF_{01}$.

\begin{lemma}\label{lem:face_figure}
  If $\Gamma$ is acyclic then the map $\kappa$ is a poset isomorphism from $\cF(Q(\Gamma,\0))$ to the face figure
  $\cF_{01}$ of the face $F_{01}$ of the order polytope $\Ord(\Gamma)$.
\end{lemma}
\begin{proof}
  For any face $G\in\cF_{01}$ let $\lambda(G)$ be the cone $\pos(G)+\RR\1$. Since $G$ is a face which is not contained
  in any facet of $[0,1]^k$ it is the intersection of facets of type $x_i\leq x_j$.  These inequalities are homogeneous,
  and so they also hold for $\lambda(G)$.  Those inequalities are tight for $Q(\Gamma,\0)$, and so $\lambda$ defines a map
  from $\cF_{01}$ to $\cF(Q(\Gamma,\0))$.  This also shows that, for any face $F$ of $Q(\Gamma,\0)$ we have
  $\lambda(\kappa(F))=F$ which means that $\kappa$ is one-to-one.  Conversely, let $G$ be a face of $\Ord(\Gamma)$ which
  is contained in $\cF_{01}$.  Then $G$ is defined in terms of split equations of the form $x_i=x_j$.  These equations
  are valid for $\lambda(G)=\pos(G)+\RR\1$, which yields $\kappa(\lambda(G))=G$.  Hence $\kappa$ is surjective, and
  $\lambda$ is the inverse map.
\end{proof}

Stanley gave a concise description of the face lattices of order polytopes in terms of partitions \cite[Theorem 1.2]{Stanley:1986}, and this can be used to derive the following result.
This should be compared with \cite[Proposition 3.5]{PostnikovReinerWilliams08} which also characterizes the faces of the braid cones, but in a different language.
\begin{theorem}\label{thm:partition}
  Let $\Gamma$ be an arbitrary directed graph on the node set $[k]$.
  Then a partition $E$ of $[k]$ is the equality partition of a face of the braid cone $Q(\Gamma,\0)$ if and only if 
  \begin{enumerate}[label = (\roman*)]
  \item for each part $K$ of $E$ the induced subgraph of $\Gamma$ on $K$ is weakly connected, and
  \item the minor of $\Gamma$ which results from simultaneously contracting each part of $E$ does not contain any directed cycle.
  \end{enumerate}
\end{theorem}
\begin{proof}
  Let us first assume that $\Gamma$ is acyclic.  By Lemma~\ref{lemma: face of wdp}, together with the fact that every cycle has weight zero, the faces of $Q(\Gamma,\0)$ are given in terms of the equality partitions of $[k]$.  In the acyclic case Lemma~\ref{lem:face_figure} translates faces of
  $Q(\Gamma,\0)$ into faces of the order polytope $\Ord(\Gamma)$ which contain the special face $F_{01}$.  The property
  (i) is the connectedness, and property (ii) is the `compatibility' condition in Stanley's result \cite[Theorem
  1.2]{Stanley:1986}.

  We now turn to the general case.  If $\Gamma$ has directed cycles we consider its \emph{acyclic reduction}. The latter
  graph, occasionally also called `condensation' in the literature, is obtained by identifying the nodes in each strong
  component. Since strong components are weakly connected and gather all the directed cycles the same reasoning applies
  as before.  It is easy to see that this digraph is indeed acyclic \cite[Corollary 5]{Sharir:81}.  Each partition of
  $[k]$ which describes a face of $Q(\Gamma,\0)$ refines the partition by strong components.
\end{proof}
Notice that there are always two partitions which trivially satisfy the conditions above: The partition of $[k]$ by weak
components corresponds to the unique minimal face (which is the lineality space); the partition by strong components
corresponds to the entire cone.

\begin{figure}[ht]
  \centering
  \resizebox{0.4\textwidth}{!}{\inclpic{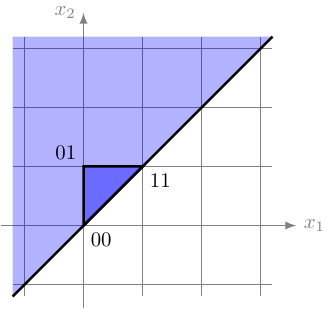}}
  \caption{Braid cone of a single arc and the corresponding order polytope; see Example~\ref{example:order_polytope}}
  \label{figure:cone_order_polytope}
\end{figure}

\begin{figure}[ht]
  \centering
 \resizebox{0.3\textwidth}{!}{\inclpic{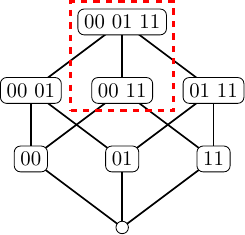}}
  \caption{Hasse diagram of the triangle $\conv\{00,01,11\}$ with face figure of $\conv\{00,11\}$ marked}
  \label{figure:digraph_order_polytope}
\end{figure}

\begin{example}\label{example:order_polytope}
  The smallest non-trivial case is $k=2$, and $\Gamma$ is the directed graph with two nodes, labeled $1$ and $2$, with
  one arc from $1$ to $2$.  The order polytope is the triangle $\conv\{00,01,11\}$, and the face $F_{01}$ is the edge
  from $00$ to $11$.  The braid cone $Q(\Gamma,\0)$ is the linear half-space $x_1 \leq x_2$, and its lineality space is
  $\RR\1$.  The braid cone and the order polytope are shown in Figure~\ref{figure:cone_order_polytope}.  The node set
  of $\Gamma$ only admits the two trivial partitions.  The Hasse diagram of the face lattice of $\Ord(\Gamma)$ and the
  face figure $\cF_{01}$ are displayed in Figure~\ref{figure:digraph_order_polytope}.
\end{example}

\begin{example}
  Figure \ref{figure: unweighted graph representing cone} shows a digraph on eight nodes and its acyclic reduction,
  which has six nodes.  Figure \ref{figure: Hasse-Diagram cone} shows the Hasse diagram of the braid cone.
  That cone is $6$-dimensional with a $1$-dimensional lineality space.  Modulo its lineality space every cone is
  projectively equivalent to a pyramid over its face at infinity.  In this case the braid cone inherits the
  combinatorics of a $4$-simplex.
\end{example}

\begin{figure}[ht]
  \centering
  \inclpic{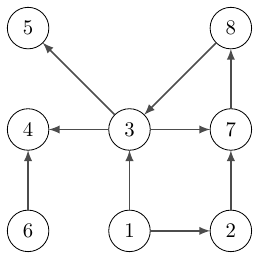}
  \hspace*{4em}
  \inclpic{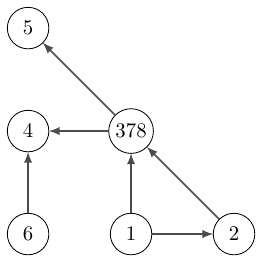}
  \caption{Digraph (left) and its acyclic reduction (right)} 
  \label{figure: unweighted graph representing cone}
\end{figure}

\begin{figure}[ht]
  \centering
  \resizebox{0.75\textwidth}{!}{\inclpic{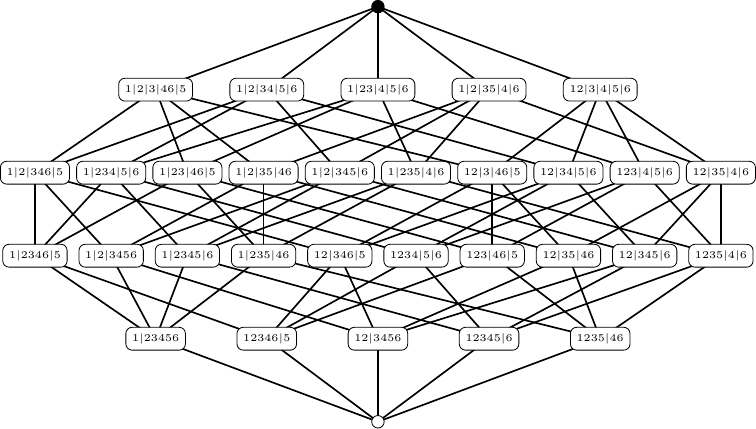}}
  \caption{Hasse diagram of the braid cone corresponding to the graph in Figure \ref{figure: unweighted graph
      representing cone}.  For improved readability the node $378$ of the acyclic reduction is represented as $3$}
  \label{figure: Hasse-Diagram cone}
\end{figure}

\begin{remark}
  Two distinct digraphs on the node set $[k]$ may induce the same braid cone.  This is the case if and only if they
  induce the same poset.  For instance, in Figure \ref{figure: unweighted graph representing cone} the arc $(1,3)$ in the graph on the
  left and the arc $(1,378)$ in the graph on the right are redundant.  In the acyclic reduction (on the right) we obtain
  a tree with directed edges.  Every tree on $\ell$ nodes has $\ell-1$ edges, and the braid cone is a simplex cone of
  dimension $\ell-1$.
\end{remark}

\subsection{Weyl--Minkowski decomposition}
Now we want to use the Theorem~\ref{thm:partition} on braid cones to describe digraph polyhedra for arbitrary weights.
Again we pick a $k{\times}k$-matrix $W$, and we assume that $Q(W)$ is feasible. The classical theorem of Weyl and
Minkowski (cf. \cite[\S 1]{Ziegler95}) states that any ordinary polyhedron $Q$ decomposes as the Minkowski sum
\begin{equation}\label{eq:Weyl-Minkowski}
  Q \ = \ P + L + C \enspace ,
\end{equation}
where $P$ is a polytope, $L$ is a linear subspace and $C$ is a pointed polyhedral cone.  An ordinary polyhedral cone is
\emph{pointed} if it does not contain any affine line (and thus no affine subspace of positive dimension).  In the
decomposition~\eqref{eq:Weyl-Minkowski} the maximal linear subspace $L$ is unique, while, in general, there may be many
choices for $C$ and $P$.  The \emph{recession cone} (which is again unique) is the Minkowski sum of the two unbounded
parts, $L$ and $C$.  The \emph{pointed part} is the Minkowski sum $P+C$ (which is unique up to an affine
transformation).  Next we will decompose a weighted digraph polyhedron in this fashion. We decompose $W$ into the graph
$\Gamma$ and the weight function $\gamma$ such that $W=W(\Gamma,\gamma)$.

\begin{lemma} \label{lemma:recession_cone} The recession cone of the weighted digraph polyhedron $Q(\Gamma,\gamma)$ is
  the braid cone~$Q(\Gamma,\0)$, and $Q(W(\Gamma,\0)\#\Gamma)$ forms the maximal linear subspace.
\end{lemma}
\begin{proof}
  Let $x$ be some point in the recession cone of $Q$. Then there exists a vector $t$ such that
  $x+\lambda t\in Q$ for all $\lambda\geq 0$.  This means that
  \[
  x_i-x_j+\lambda(t_i-t_j)\ \leq \ w_{ij}  \qquad \text{ for all } (i,j)\in \Gamma \text{ and } \lambda\geq 0 \enspace .
  \]
  This forces $t_i-t_j \leq 0$ for all $(i,j)\in \Gamma$, and so $t$ lies in $Q(\Gamma,\0)$.  The reverse inclusion is similar, and we conclude that the braid cone $Q(\Gamma,\0)$ is the recession cone of $Q$.

  Again let $t \in Q(\Gamma,\0)$.  Then its negative $-t$ is also contained in $Q(\Gamma,\0)$ if and only if
  \[
  t_i - t_j \ = \ 0 \qquad \text{for all } (i,j)\in \Gamma \enspace 
  \]
  if and only if $t \in Q(W(\Gamma,\0)\#\Gamma)$.  We infer that the braid cone $Q(W(\Gamma,\0)\#\Gamma)$ forms
  the maximal linear subspace of $Q$.
\end{proof}

As a corollary we obtain a slight generalization of \cite[Corollary 12]{DevelinSturmfels:2004}.

\begin{corollary}\label{coro:boundedness_wdp}
  The weighted digraph polyhedron $Q(\Gamma, \gamma)$ is bounded in $\RR^d / \RR\1$ if and only if $\Gamma$ consists of one strong component.
\end{corollary}
\begin{proof}
If $\Gamma$ has only one strong component, then the recession cone $(\Gamma, \0)$ is exactly the one-dimensional lineality space $\RR\1$ by Proposition~\ref{prop:cones_components}.  Hence, $Q(\Gamma, \gamma)$ is bounded in $\RR^d / \RR\1$.
Otherwise, the recession cone is higher-dimensional and the weighted digraph polyhedron is unbounded.
\end{proof}

Our next goal is to describe a minimal system of generators for a braid cone.  Recall that a pointed cone is
projectively equivalent to a pyramid over its far face.  The minimal generators of a pointed cone correspond to
the vertices of the far face.  For any subset $K \subseteq [k]$, let $\chi(K)\in \RR^k$ be the characteristic vector.
That is, the $i$th coordinate of $\chi(K)$ is one if $i\in K$, and it is zero otherwise.  With this notation, e.g., we
have $\chi([k])=\1$ and $\chi(\emptyset)=\0$.

\begin{proposition}\label{prop:rays_of_cone}
  A minimal system of generators of the pointed part of the braid cone $Q(\Gamma,\0)$ is given by the vectors
  $\chi(K)$ with $K\subseteq [k]$ so that the induced subgraph on $K$ is connected, its complement in its weak component in $\Gamma$ 
  is also connected and every arc in the cut-set of this partition is directed from $[k]\setminus K$ to $K$.
\end{proposition}
\begin{proof}
  Let $K_1, \ldots, K_{\ell}$ be the weak components of $\Gamma$.
  In particular, by applying Proposition~\ref{prop:cones_components} to $Q(W(\Gamma,\0)\#\Gamma)$, the dimension of the lineality space of $Q(\Gamma,\0)$ equals $\ell$.
  Let $F$ be a minimal non-trivial face of the cone $Q(\Gamma,\0)$.
  This is a Minkowski sum of the lineality space with a single ray.
  By Theorem~\ref{thm:partition} the latter corresponds to a partition with $\ell + 1$ parts.
  Among these exactly $\ell-1$ parts are weak components of $\Gamma$, while the remaining weak component is split into two.
  Let us assume that the remaining component decomposes as $K_u = K \cup (K_u\setminus K)$, where every arc in the cut-set is directed from $K_u \setminus K$ to $K$.
  The characteristic vectors $\chi(K_i)$ for $i \in [\ell]$ linearly span the lineality space of $Q(\Gamma,\0)$, while $\chi(K)$ generates the pointed part of $F$.
\end{proof}

\subsection{Envelopes and duality}
We now turn to the construction of a special class of digraph polyhedra which were introduced by Develin and Sturmfels
for studying tropical convexity from the viewpoint of geometric combinatorics \cite{DevelinSturmfels:2004}.  For a
$d{\times}n$-matrix $V$ with coefficients in $\TT_{\min}=\RR\cup\{\infty\}$ we look at the ordinary polyhedron
\begin{align*}
  \envelope(V) \ &= \ \SetOf{(y,z)\in\RR^d\times\RR^n}{y_i - z_j \leq v_{ij} \text{ for all } i\in[d] \text{ and } j\in[n]}\\
                 &= \ \SetOf{(y,z)\in\RR^d\times\RR^n}{y_i - z_j \leq v_{ij} \text{ for all } (i,j)\in \Bgraph} \enspace ,
\end{align*}
where
\begin{equation}\label{eq:finite}
  \Bgraph(V) \ = \ \SetOf{(i,j)\in [d]\times[n]}{v_{ij} \neq \infty}
\end{equation}
is a (bipartite) directed graph recording the finite entries of $V$.  We call $\envelope(V)$ the \emph{envelope} of the
matrix $V$.  We may see the envelope as a weighted digraph polyhedron via the matrix $(d+n)\times(d+n)$-matrix $W$
which is defined as
\begin{equation}\label{eq:W}
W \ =\ 
\left(
\begin{array}{cc}
  \infty_{d\times d} & V \\
  \infty_{n\times d} & \infty_{n\times n} 
\end{array} 
\right) \enspace .
\end{equation}
Up to an obvious relabeling of the nodes $\Bgraph(V)$ is the same as $\Gamma(W)$ for the matrix $W$ defined above, and thus we
can identify $\envelope(V)$ with $Q(W)$.  Applying Lemma \ref{lemma:recession_cone} and Proposition \ref{prop:rays_of_cone} to the envelope we obtain the following.
\begin{corollary}\label{coro:generators_cone_envelope}
  The minimal generators of the pointed part of the recession cone of the envelope are given by the partitions $D' \dcup D''=[d]$ and $N' \dcup N''=[n]$  so that 
  \begin{enumerate}[label = (\roman*)]
  \item the induced subgraph on $D'{\dcup}N'$ has the same number of weak components as $\Bgraph$,
  \item \label{item:connected_compo_generator} the induced subgraph on $D''{\dcup}N''$ is connected, and
  \item there are no arcs from $D''$ to $N'$.
  \end{enumerate}
  The characteristic vector of $D''{\dcup}N''$ now yields one such generator. 
\end{corollary}
Similarly we obtain from Proposition~\ref{prop:rays_of_cone} the following corollary which will be helpful in section~\ref{sec:projective}.
A ray can be scaled modulo $\1$ so that it has only non-negative entries and at least one zero entry.
Then the \emph{support} of the ray is the set of indices of the non-zero entries.
We keep the notation of the former corollary and consider a face of the envelope $\envelope(V)$ defined by the graph $G$ that contains a minimal generator with support $D''{\dcup}N''$. Notice that the arcs of $\Gamma(W\#G)$ which are not arcs of $B$ are arcs from $N'$ to $D'$ or from $N''$ to $D'$ or from $N''$ to $D''$. That is, there are no arcs from $N'$ to $D''$.

\begin{corollary}\label{coro:inf_env}
  Let $M$ be the set of column indices $j$ of the matrix $V$ such that $v_{ij} = \infty$ for all $i \in D''$.
  Then $M$ equals $N'$, and none of the shortest paths in $\Gamma(W\#G)$ between any two nodes in $D'$ contains a node in $D'' \dcup ([n] \setminus M)$.
\end{corollary}
\begin{proof}
  Observe that $M$ is exactly the subset of the nodes in $[n]$ without an arc between $D''$ and $M$ in $\Gamma(W\#G)$.
  Hence, we obtain $N' \subseteq M$ and with Corollary \ref{coro:generators_cone_envelope}\ref{item:connected_compo_generator} even $N' = M$.
  This yields $[n] \setminus M = N''$. Hence, by Proposition~\ref{prop:rays_of_cone}, there is no arc from $[n] \setminus M$ to $D'$ in $\Gamma(W\#G)$.
  This implies that every shortest path between two nodes in $D'$ avoids the set $D'' \dcup ([n] \setminus M)$.
\end{proof}

The graph $\Bgraph(V)$ has two kinds of nodes, those which correspond to the rows and those which represent columns of
$V$. In our drawings, like Figure~\ref{figure:bipartite_graph_recession_cone}, we show row nodes as rectangles and
column nodes as circles.  Moreover, we always draw the row nodes above the column nodes.  Therefore, if we want to
distinguish them we sometimes talk about the \emph{top} and the \emph{bottom shore} of the bipartite graph.

\begin{figure}[ht]
  \centering
  \resizebox{0.25\textwidth}{!}{\inclpic{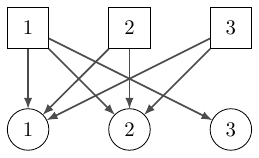}}
  \caption{Bipartite digraph $\Bgraph(V)$ for the matrix in Example \ref{example:envelope_with_infty}} 
  \label{figure:bipartite_graph_recession_cone}
\end{figure}

\begin{example}\label{example:envelope_with_infty}
%
%
%
  For $d=n=3$ consider the $3{\times}3$-matrix
  \[
  V \ = \ \begin{pmatrix} 0 & 0 & 0 \\ 1 & 1 & \infty \\ 0 & 2 & \infty \end{pmatrix} \enspace .
  \]
  The lineality space of the envelope $\envelope(V)$ is spanned by $\1$.  The quotient $\envelope(V)/\1$ is
  $5$-dimensional, and it has exactly two vertices: ${(0,1,0;0,0,0)}$ and ${(0,1,2;2,0,0)}$.  Its recession cone has six
  minimal generators, which arise from partitioning the bipartite graph $\Bgraph(V)$, which is a subgraph of $K_{3,3}$,
  into two induced subgraphs which meet the criteria of Corollary~\ref{coro:generators_cone_envelope}, see Figure~\ref{figure:bipartite_graph_recession_cone}.  The sets of the
  form $D'' \dcup N''$ read
  \[
  \emptyset{\dcup}1 \,, \quad \emptyset{\dcup}2 \,, \quad \emptyset{\dcup}3 \,,
  \quad 12{\dcup}123 \,, \quad  13{\dcup}123 \,, \quad 23{\dcup}12 \enspace .
  \]
  The complementary parts are given by $D'=\{1,2,3\}\setminus D''$ and $N'=\{1,2,3\}\setminus N''$.  Notice that, e.g.,
  $23{\dcup}123$ does not occur in the list above since $v_{23}=\infty=v_{33}$; this implies that the induced subgraph
  is not connected.  For instance, $23{\dcup}12$ yields the generator ${(0,1,1;1,1,0)}$.
\end{example}

A \emph{subpolytope} of a polytope $P$ is the convex hull of some subset of the vertices of $P$.  Each face is a
subpolytope, but the converse does not hold.  We write $e_i$ for the $i$th standard basis vector of $\RR^k$, for any
$k$, and we write vectors in the product space $\RR^d\times\RR^n$ as $\pair{x}{y}$ where $x\in\RR^d$ and $y\in\RR^n$.
With this notation
\[
\Delta_{d-1}\times\Delta_{n-1} \ = \ \conv\SetOf{\pair{e_i}{e_j}}{(i,j)\in[d]\times[n]}
\]
is a product of simplices. Develin and Sturmfels established that a tropical configuration of $n$ points induces a
polyhedral subdivision of $\RR^d$ which is dual to a regular subdivision of $\Delta_{d-1}\times\Delta_{n-1}$
\cite[Theorem~1]{DevelinSturmfels:2004}.  A polytopal subdivision is \emph{regular} if it is induced by a height
function; for details see \cite{Triangulations}.  The following statement will be instrumental in Section~\ref{subsec:type}
below for obtaining a natural generalization to subpolytopes of products of simplices.  Notice that those subpolytopes
naturally correspond to subgraphs of the complete bipartite graph $[d]\times[n]$. 
\begin{theorem}\label{thm:regular_subdivison}
  The boundary complex of the envelope $\envelope(V)$ is dual to the regular subdivision of the polytope
  \[
  \conv \SetOf{\pair{e_i}{e_j} \in \RR^d \times \RR^n}{(i,j) \in \Bgraph(V)}
  \]
  with height function $V$.
\end{theorem}
\begin{proof}
  We abbreviate $\Bgraph=\Bgraph(V)$. Homogenizing the envelope $\envelope(V)$ (with leading homogenizing coordinate) yields the cone
  \[
  \SetOf{\triple{\alpha}{y}{z} \in \RR_{\geq 0}\times\RR^d\times\RR^n}{ \langle \triple{v_{ij}}{-e_i}{e_j}, \triple{\alpha}{y}{z} \rangle
  \geq 0 \text{ for all } (i,j) \in \Bgraph} \enspace .
  \]
  Hence the polar cone with the dual face lattice can be written as
  \[
  \pos\left\{\triple{1}{\0}{\0}\right\} + \pos \SetOf{\triple{v_{ij}}{-e_i}{e_j}}{(i,j) \in \Bgraph} \enspace .
  \]
  Intersecting with the affine hyperplane $H=\smallSetOf{\triple{\alpha}{y}{z}}{\langle \triple{0}{-\1}{\1}, \triple{\alpha}{y}{z}
    \rangle = 2}$ gives the polytope
  \[
  P \ = \ \conv\SetOf{\triple{v_{ij}}{-e_i}{e_j}}{(i,j) \in \Bgraph} \enspace ,
  \]
  because all these vectors lie in $H$ and the origin does not. 
  
  The orthogonal projection of the lower convex hull of $P$ with respect to $\triple{1}{\0}{\0}$ defines a regular
  subdivision of the subpolytope of $\Delta_{d-1}\times\Delta_{n-1}$ corresponding to $\Bgraph$.  If $\Bgraph$ is the
  complete bipartite graph or equivalently no entry of $V$ is $\infty$, that subpolytope is the entire product of
  simplices.
\end{proof}

Any regular subdivision of a subpolytope extends to a regular subdivision of the superpolytope, e.g., by successive
placing of the remaining vertices \cite[\S4.3.1]{Triangulations}. In our situation a regular subdivision of the
superpolytope $\Delta_{d-1}\times\Delta_{n-1}$ is obtained by replacing the infinite coefficients in the matrix $V$ with
sufficiently large real numbers.  Note that this extension is not unique.

\subsection{Projections} In this section we investigate orthogonal projections of weighted digraph polyhedra and
envelopes into the coordinate directions.  To this end we let $\pi_I$ be the projection onto the coordinates in
$[k]\setminus I$ for $I\subseteq[k]$.  For a $k{\times} k$-matrix $W$ we define $W/I$ by removing the rows and columns
whose indices lie in $I$.  We write $\pi_i$ and $W/i$ if $I=\{i\}$ is a singleton.
\begin{lemma} \label{lemma:projection_wdp}
  The image of $Q(W)=Q(W^*)$ under the linear projection $\pi_I$ is the weighted digraph polyhedron $Q(W^*/I)$.
\end{lemma}
\begin{proof}
  By induction it suffices to consider the case where $I=\{k\}$.
  That $\pi_{k}(Q(W^*))$ is contained in $Q(W^*/k)$ is clear.
  We want to show the reverse inclusion.
  For $(x_1, \ldots, x_{k-1}) \in Q(W^*/k)$ we need to find a real number $y$ so that $(x_1,\dots,x_{k-1} ,y ) \in Q(W)=Q(W^*)$.
  The latter condition is equivalent to
  \[
  x_i - w_{ik}^* \ \leq \ y \quad \text{ and } \quad y \ \leq \ x_i + w_{ki}^* \quad \text{ for all } \quad i \in [k-1] \enspace .
  \]
  So, the claim follows if we can show that 
  \begin{equation}\label{eq:maxmin}
    \max_{i \in [k-1]} (x_i - w_{ik}^*) \ \leq \ \min_{i \in [k-1]}(x_i + w_{ki}^*) \enspace .
  \end{equation}
  Let $p$ and $q$ be indices for which the maximum and the minimum in \eqref{eq:maxmin}, respectively, are attained.
  Now $w^*_{pq}$ is the length of the shortest path from $p$ to $q$ in the weighted digraph $\Gamma(W)$.  This yields
  \[
  x_p - x_q \ \leq \ w^*_{pq} \ \leq \ w_{pk}^* + w_{kq}^*  \quad\text{and hence}\quad x_p - w_{pk}^* \ \leq \ x_q + w_{kq}^*  \enspace .
  \]
\end{proof}

Now we turn to studying projections of faces of the envelope $\envelope(V)$ of a not necessarily square $d{\times}n$-matrix.
With $W$ defined as in \eqref{eq:W} we have $\envelope(V)=Q(W)$.  By Lemma~\ref{lemma: face of wdp} for any face $F$ of
the envelope there is a subgraph $G$ of $\Gamma=\Gamma(W)$ such that $F=Q(W\#G)$.  Since, up to a relabeling of the
nodes, we can identify the directed graph $\Gamma$ with the bipartite graph $\Bgraph=\Bgraph(V)$ and we may read $G$ as a
subgraph of $\Bgraph$.  We define the $n{\times}d$-matrix $V[G]$ with coefficients
\[
v_{ji}' \ = \ \begin{cases} -v_{ij} & \text{if } (i,j)\in G\\ \infty & \text{otherwise \enspace .} \end{cases}
\]
The following lemma is similar to \cite[Lemma~10]{DevelinSturmfels:2004}.  Notice that the tropical matrix product $V
\odot V[G]$ yields a $d{\times d}$-matrix.
\begin{lemma}\label{lemma: projection envelope}
  The image of the face $F$ of $\envelope(V)\subset\RR^d\times\RR^n$ under the orthogonal projection $\pi_{[n]}$ onto
  the first component is the weighted digraph polyhedron $Q(V \odot V[G])$.
\end{lemma}
\begin{proof}
  For $i,\ell\in[d]$ let $u_{i\ell}$ be a coefficient of $V\odot V[G]$. We have 
  \[
  u_{i\ell} \ = \ \min_{j\in[n]}(v_{ij}+v_{j\ell}') \ = \ \min_{j\in[n],\, v_{ij}\neq\infty,\, v_{\ell
      j}\neq\infty}(v_{ij}-v_{\ell j}) \enspace ,
  \]
  which is exactly the length of a shortest path from $i$ to $\ell$ with two arcs in the digraph $\Gamma(W\#G)$. Since
  the directed graph $\Gamma(W\#G)$ is bipartite the shortest path from $i$ to $\ell$ (over arbitrarily many arcs) is a
  concatenation of the two-arc-paths above.  Now the claim follows from the previous lemma.
\end{proof}

\begin{figure}[ht]
  \centering
  \inclpic{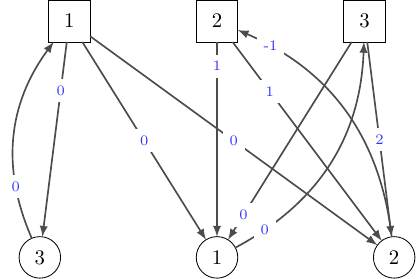}
\hspace*{4em}
  \inclpic{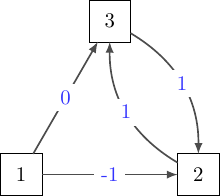}
  \caption{Weighted digraphs corresponding to a face of $\envelope(V)$ from Example \ref{example:projection_face_envelope}. The first graph corresponds to a face in $\RR^d{\times}\RR^n$ whereas the second corresponds to its projection onto $\RR^d$.  The nodes on the bottom shore are not in their natural ordering to reduce the number of arcs crossing} 
  \label{figure:face_envelope}
\end{figure}

\begin{example} \label{example:projection_face_envelope}
We consider the same matrix $V$ as in Example~\ref{example:envelope_with_infty}. For the bipartite graph $G$ on the six nodes $\{1,2,3\} \dcup \{1,2,3\}$ with arcs $(1,3),(2,2),(3,1)$ we obtain
\[
  V[G] \ = \ \begin{pmatrix} \infty & \infty & 0 \\ \infty & -1 & \infty \\ 0 & \infty & \infty \end{pmatrix} \enspace .
\]
This yields the product 
\[
  V\odot V[G] \ = \ \begin{pmatrix} 0 & 0 & 0 \\ 1 & 1 & \infty \\ 0 & 2 & \infty \end{pmatrix} \odot
 \begin{pmatrix} \infty & \infty & 0 \\ \infty & -1 & \infty \\ 0 & \infty & \infty \end{pmatrix} \ = \ \begin{pmatrix} 0 & -1 & 0 \\ \infty & 0 & 1 \\ \infty & 1 & 0 \end{pmatrix} \enspace .
\]
The corresponding graph is depicted in Figure~\ref{figure:face_envelope} on the right whereas the left one shows the graph $\Gamma(W\#G)$.
\end{example}

\begin{figure}[ht]
  \centering
  \resizebox{0.4\textwidth}{!}{\inclpic{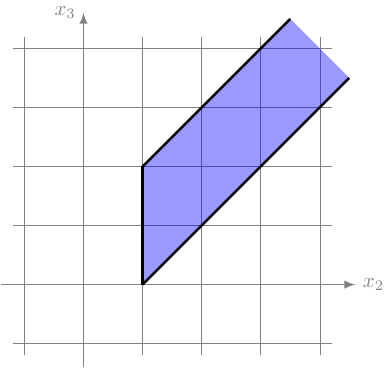}}
  \caption{Weighted digraph polyhedron given by the matrix $V \odot V[G]$ in Example
    \ref{example:projection_face_envelope} which is unbounded in the tropical projective $2$-torus}
  \label{figure:infinite_convex_hull}
\end{figure}

\section{Tropical cones and polyhedral cells} \label{sec:tropical}
\subsection{Polyhedral sectors}
As before let $V$ be a $d{\times}n$-matrix with coefficients in $\TT_{\min}$. We write $v^{(j)}$ for the $j$th column of $V$, and therefore we can identify $V$ with $(v^{(1)},v^{(2)},\dots,v^{(n)})$, the sequence of column vectors.  The
$(\min,+)$-linear span of the columns of $V$ is the \emph{$\min$-tropical cone}
\[
\tcone(V) \ = \ \SetOf{(\lambda_1 \odot v^{(1)}) \oplus \dots \oplus (\lambda_n \odot v^{(n)})}{\lambda_j\in\TT_{\min}} \enspace .
\]
Put in a more algebraic language, a tropical cone is the same as a finitely generated subsemimodule of the semimodule
$(\TT_{\min}^d,\oplus,\odot)$.  A subset $M$ of $\RR^d$ is \emph{$\min$-tropically convex} if for any two points $u,v\in
M$ we have $\tcone(u,v)\subseteq M$.  Any tropically convex set contains $\RR\1$, and so we can study its image under
the canonical projection to the tropical projective torus.  Up to this projection tropical cones generated by
vectors with finite entries are precisely the `tropical polytopes' of Develin and Sturmfels \cite{DevelinSturmfels:2004}.
In this section we will generalize key results from that paper to the case where $\infty$ may occur as a coordinate.
By homogenization our results also apply to the formally more general `tropical polyhedra' studied, e.g., in \cite{AkianGaubertGutermann12} and \cite{ABGJ-Simplex:A}.

\begin{remark}\label{rem:wdp_min_cone}
  For an arbitrary $k{\times}k$-matrix with coefficients in $\TT_{\min}$ the weighted digraph polyhedron
  $Q(W)=Q(W^*)$ coincides with the $\min$-tropical span $\tcone(W^*)$. See also \cite[Theorem~2.1.1]{Butkovic:10} and the Section~\ref{sec:polytropes} on
  polytropes below.
\end{remark}

For $u \in \TT_{\min}^d$ and $i \in [d]$ with $u_i \not= \infty$ we define the $i$th \emph{sector} $S_i(u)$ with
respect to \emph{max} as
\[
\SetOf{z \in \RR^d}{\max_{\ell \in [d]}(z_{\ell} - u_{\ell}) = z_i - u_i} \ = \ \SetOf{z \in \RR^d}{\min_{\ell \in [d]}(u_{\ell} - z_{\ell}) = u_i - z_i} \enspace .
\]
Notice that the above equality of sets is a consequence of the elementary fact
\[
-\max(u,v) \ = \ \min(-u,-v) \enspace .
\]
Moreover, the equation $\min_{\ell \in [d]}(u_{\ell} - z_{\ell}) = u_i - z_i$ is equivalent to $z_{\ell} - z_i \leq
u_{\ell} - u_i$ for each $\ell \in [d]$.  As $u_i<\infty$ that minimum cannot be attained for any $\ell \in [d]$ with
$u_{\ell} = \infty$.  We have
\begin{equation}\label{eq:sector}
  S_{i}(u) \ = \ \bigcap_{\ell\in[d],\ u_\ell\neq\infty}\SetOf{z\in\RR^d}{z_{\ell} - z_i \leq u_{\ell} - u_i}
\enspace ,
\end{equation}
which means that this sector is the weighted digraph polyhedron for the graph with node set $[d]$ and arc set
$\smallSetOf{(\ell,i)}{\ell \in [d], u_{\ell} \not= \infty}$, where the arc $(\ell,i)$ has weight $u_{\ell} - u_i$.

\begin{figure}[ht]
  \centering
  \resizebox{0.35\textwidth}{!}{\inclpic{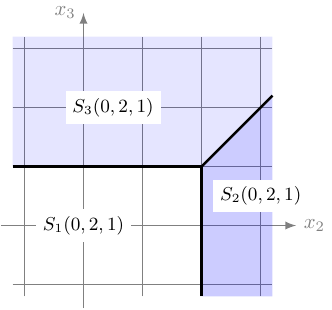}}
  \hspace*{4em}
  \resizebox{0.35\textwidth}{!}{\inclpic{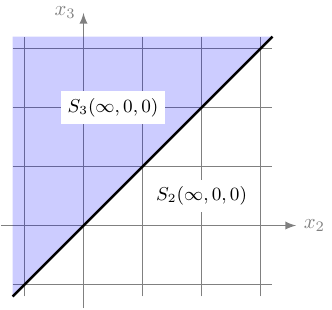}}
  \caption{Polyhedral decomposition of $\RR^3$ as in Lemma~\ref{lem:sectors} induced by $(0,2,1)$ and $(\infty,0,0)$, respectively.
    Compare the image on the right with Figure \ref{figure:cone_order_polytope}}
  \label{figure:sector_decomposition}
\end{figure}

\begin{lemma}\label{lem:sectors}
  The sectors $\smallSetOf{S_i(u)}{u_i\neq\infty}$ are the maximal cells of a polyhedral decomposition of $\RR^d$.
\end{lemma}
\begin{proof}
  Considering the column vector $u$ as a $d{\times}1$-matrix, we obtain the envelope $\envelope(u)$ as a subset of $\RR^{d+1}$.
  The sector $S_i(u)$ is the orthogonal projection of the face defined by the single arc $(i,1)$ in the bipartite graph $\Bgraph(u)$.
\end{proof}
We denote the polyhedral complex arising from the previous lemma by $\fan(u)$; see also
\cite[Proposition~16]{DevelinSturmfels:2004}. The negative $-u$ of the vector $u\in\TT_{\min}^d$ defines a
$\max$-tropical linear form and thus a $\max$-tropical hyperplane.  The sectors $S_i(u)$ for $u_i\neq\infty$ are
precisely the topological closures of the connected components of the complement of that tropical hyperplane.

\begin{example}\label{ex:bipartite_graph_sector}
  The white sector $S_1(0,2,1)$ in Figure~\ref{figure:sector_decomposition} is the orthogonal projection on $\{1,2,3\}$ of the weighted digraph polyhedron given by the bipartite graph with node set $\{1,2,3\} \dcup \{1'\}$ where the arc $(1,1')$ has weight zero, $(2,1')$ has weight $2$, $(3,1')$ has weight $1$ and $(1',1)$ has weight zero.
\end{example}

The following result characterizes the solvability of a system of tropical linear equations in $\RR^d$.  For matrices
with finite coordinates this is the Tropical Farkas Lemma \cite[Proposition~9]{DevelinSturmfels:2004}, a version of
which already occurs in \cite{Vorobyev67}.  We indicate a short proof for the sake of completeness. 
\begin{lemma}\label{lemma:sectors_tropical_hull}
  A point $z \in \RR^d$ is contained in $\tcone(V)$ if and only if for every $i \in [d]$ there is an index $s \in [n]$ with $z \in S_i(v^{(s)})$.
\end{lemma}
\begin{proof}
  Let $z \in \RR^d$ be a point in $\tcone(V)$.  Then there is a vector $\lambda \in \TT_{\min}^n$ so that
  $
  \bigoplus_{j = 1}^{n} \lambda_j \odot v^{(j)} \ = \ z 
  $
  or, equivalently,
  \begin{equation}\label{eq:sectors}
    \min\SetOf{\lambda_j + v_{ij}}{j \in [n]} \ = \ z_i \qquad \text{for each } i \in [d] \enspace.
  \end{equation}
  Now fix $i\in[d]$ and let $s$ be an index $j$ for which the minimum in \eqref{eq:sectors} is attained; that is,
  $z_i=\lambda_s+v_{is}$.  If $\ell\in[d]$ with $v_{\ell s}\neq\infty$ this gives
  \[
  z_\ell-z_i \ = \ z_\ell - \lambda_s-v_{is} \ \leq \ \lambda_j + v_{\ell j} - \lambda_s - v_{is} \qquad \text{for each }
  j\in[n] \enspace .
  \]
  Specializing to $j=s$ entails $z_\ell-z_i \leq v_{\ell s}-v_{is}$ and thus $z\in S_i(v^{(s)})$.  The entire argument
  can be reversed to prove the converse.
\end{proof}

\subsection{The covector decomposition}\label{subsec:type}
Again let $V\in\TT_{\min}^{d\times n}$, and let $W\in\TT_{\min}^{(d+n)\times(d+n)}$ be the matrix which is associated
via \eqref{eq:W}. We assume in the following that $V$ has no column equal to the all $\infty$ vector $\transpose{(\infty,\ldots,\infty)}$; hence, none of the complexes $\fan(v^{(j)})$ is empty.
We do admit rows which solely contain $\infty$ entries.
They add to the lineality of the occurring polyhedra.
However, there may also be other contributions to the lineality space; see Lemma~\ref{lemma:recession_cone}.
The weighted bipartite graph $\Bgraph=\Bgraph(V)$ and the weighted digraph $\Gamma=\Gamma(W)$ are defined as before. For an arbitrary subgraph $G$ of $\Bgraph$ we define the polyhedron
\begin{equation}\label{eq:def_type_cell}
  X_G(V) \ = \ \bigcap_{(i,j)\in G} S_i(v^{(j)})
\end{equation}
in $\RR^d$.

\begin{remark}\label{rem:type_relations}
  Right from the definition, we obtain $X_{G \cup H}(V) = X_G(V) \cap X_H(V)$ for any two graphs $G,H \subseteq
  \Bgraph(V)$ . If, furthermore, $G \subseteq H$ then $X_H(V) \subseteq X_G(V)$.  This occurs also in \cite[Corollary 11
  and 13]{DevelinSturmfels:2004}.
  It should be stressed that the cells $X_G(V)$ and $X_H(V)$ may coincide even if the graphs $G$ and $H$ are distinct.
\end{remark}

\begin{proposition} \label{prop:projection_face_sector}
  Let $G$ be an arbitrary subgraph of $\Bgraph$ (which we may also read as a subgraph of~$\Gamma$).  Then the orthogonal
  projection of the face $F_G(W)$ onto $\RR^d$ equals $X_G(V)$.
  If no node in $[n]$ is isolated in $G$ that projection is an affine isomorphism.
\end{proposition}
\begin{proof}
  Our goal is to exploit what we know about weighted digraph polyhedra.  To this end we define several digraphs with the
  same node set $[d]\dcup G$. Recall that we identify the subgraph $G$ of $\Gamma$ with its set of edges.  However, in
  the class of digraphs to be defined now, those edges (along with the nodes in $[d]$) play the role of nodes.

  Pick $(i,j)\in G$.  We let $\Phi_{ij}$ be the weighted digraph which results from $\Bgraph(v^{(j)})$, which has
  $[d]\dcup \{1\}$ as its node set, by renaming the node $1$ on the bottom shore by $(i,j)$ and adding an isolated node for each other arc in $G$.  The graph $\Phi_{ij}$
  has one extra arc in the reverse direction, namely from $(i,j)$ to $i$.  The weights on the arcs from top to bottom
  are the same as in $\Bgraph(v^{(j)})$, while the weight on the single reverse arc is $-v_{ij}$.  Compare this with Lemma \ref{lem:sectors} and Example \ref{ex:bipartite_graph_sector}. By construction the
  weighted digraph $\Phi_{ij}$ is bipartite and thus can be identified with a square matrix of size $d+|G|$.  By Lemma
  \ref{lemma: projection envelope} the weighted digraph polyhedron $Q(\Phi_{ij})\subset\RR^d\times\RR^G$ projects
  orthogonally onto the sector $S_i(v^{(j)})\subset\RR^d$.

  Let $\Phi$ be the digraph with node set $[d]\dcup G$ which is obtained as the union of the digraphs $\Phi_{ij}$ for
  $(i,j)\in G$.  Notice that by our construction the choice of the weights for the individual graphs $\Phi_{ij}$ is
  consistent.  This way we obtain a natural weight function on $\Phi$.  Due to Lemma~\ref{lemma: intersection wgp} we
  have
  \[
  \pi_G\bigl(Q(\Phi)\bigr) \ = \ \pi_G\bigl(\bigcap_{(i,j)\in G} Q(\Phi_{ij})\bigr) \ = \ \bigcap_{(i,j)\in G} S_i(v^{(j)}) \enspace .
  \]
  If $\Gamma(W\#G)$ has a negative cycle, so has $\Phi$ and by Lemma~\ref{lemma:feasible} then $F_G(W)$ as well as $X_G(V)$ are empty.
  If there are no negative cycles, there exists a shortest path between two nodes $i$ and $\ell$ in $[d]$, and it does not matter if we consider $\Gamma(W\#G)$ or $\Phi$. So, the claim follows with Lemma~\ref{lemma:projection_wdp}. 

  For the rest, assume that $\Gamma(W\#G)$ has no negative cycle.
  Since $\Gamma(W\#G)$ is bipartite, any two nodes $i, \ell \in [d]$ are contained in a directed cycle of weight zero of $G$ if this also holds for the graph
  $\Gamma(\pi_{[n]}(Q(W\#G))$ of the projection of $F_G(W)$ by Lemma~\ref{lemma:projection_wdp}. If no node in $[n]$ is isolated in $G$, every node in $[n]$ is contained in a directed cycle of weight zero, as every arc from $[n]$ to $[d]$ in $\Gamma(W\#G)$ induces a cycle 
  of length zero. Hence, the equality partition of $F_G(W)$ and of its projection $\Gamma(\pi_{[n]}(Q(W\#G))$ have the same number of parts by Lemma~\ref{lemma:zero_cycles}\ref{lemma:item:zero_cycles}. Therefore, if no node in $[n]$ is isolated in $G$, we get
  that $F_G(W)$ has the same dimension as $X_G(V)$.
\end{proof}

The \emph{covector decomposition} $\typedecomp{V}$ of $\RR^d$ is the common refinement of the polyhedral complexes
$\fan(v^{(j)})$ for $j \in [n]$.  For every cell $C$ in the covector decomposition there is a unique maximal subgraph
$\typegraph{C}$ of the complete bipartite graph $[d]\times[n]$, called the \emph{covector graph} of $C$, such that $C =
X_{\typegraph{C}}(V)$.  This graph is equivalent to the \emph{covector} $(t_1,t_2,\ldots,t_d) \in [n]^d$ where $t_i
\subseteq [n]$ consists of the nodes adjacent to $i$.  While the covector notation is concise in most proofs it is
convenient to keep the interpretation as a directed bipartite graph.  Notice that our cells are closed by definition.
By Proposition \ref{prop:projection_face_sector}, each covector (graph) also uniquely determines a face of $\envelope(V)$ and every face, for which
no node in $[n]$ is isolated, occurs in this way.  By Lemma \ref{lemma:sectors_tropical_hull} the \emph{covector
  decomposition} $\typedecomp{V}$ of $\RR^d$ induces a covector decomposition of the tropical cone $\tcone(V)$.
The covector graphs correspond to the `types' of \cite{DevelinSturmfels:2004}.

\begin{example}
  \label{ex:type_decomposition}
  Figure \ref{figure:tropical_types} shows an example for the matrix
  \[
  V \ = \
   \begin{pmatrix}
    0 & 0 & 0 \\
    1 & 0 & \infty \\
    2 & -1 & \infty
  \end{pmatrix} \enspace .
  \]
  The points corresponding to the columns of $V$ are marked $1$, $2$ and $3$.  Notice that the third column has $\infty$
  as a coordinate, which is why this point lies outside the tropical projective torus.  In fact, it is a boundary point of the
  \emph{tropical projective plane}; see Section~\ref{sec:projective} and Figure~\ref{fig:signed_cells} below.

  Only the covectors of the full-dimensional cells are indicated since the covectors of the other cells can
  directly be deduced from them by Remark~\ref{rem:type_relations}.

  The covector decomposition of $\tcone(V)$ has precisely two cells which are maximal with respect to inclusion: the
  $2$-dimensional cell with covector $(3,2,1)$ and the $1$-dimensional cell with covector $(13,2,2) = (13,-,2) \cup (13,2,-)$.
\end{example}

\begin{figure}[ht]
  \centering
  \resizebox{0.5\textwidth}{!}{\inclpic{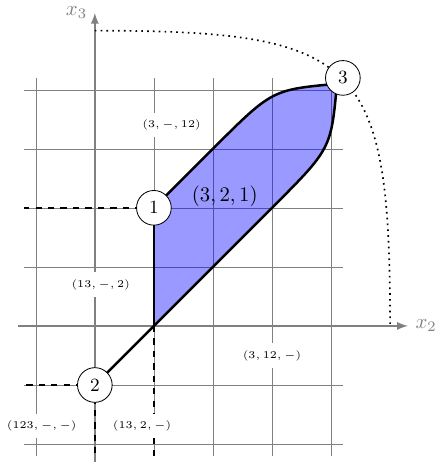}}
  \caption{Tropical cone in the tropical projective $2$-torus, from Example \ref{ex:type_decomposition}.  The dotted
    line represents the boundary, which is not part of the tropical projective torus}
  \label{figure:tropical_types}
\end{figure}

\begin{remark}
From the viewpoint of tropical geometry the decomposition  $\typedecomp{V}$ can be deduced from the $\max$-tropical linear forms corresponding to the columns of $V$. For this, we pick variables $x_{1j},x_{2j},\ldots,x_{dj}$ for each column $v^{(j)}$ of $V$. The product of the tropical linear forms $\max(x_{1j}-v_{1j},x_{2j}-v_{2j},\ldots,x_{dj}-v_{dj})$ yields a homogeneous tropical polynomial $p$ in $d \cdot n$ variables $x_{ij}$. This defines a tropical hypersurface in $\RR^{d \cdot n}/\RR\1$ where the covectors come into play as the exponent vectors of (tropical) monomials in $p$. Substituting $x_{ij}$ by $y_i$ gives rise to the tropical hypersurface in $\RR^d/\RR\1$ which induces the cell decomposition of this space.
\end{remark}

\begin{theorem}\label{thm:envelope}
  The orthogonal projection from the boundary complex of $\envelope(V)$ onto $\RR^d$ induces a bijection between the
  envelope faces whose covector graph have no isolated node in $[n]$ and the cells in the covector decomposition
  $\typedecomp{V}$ of $\RR^d$.  This map is a piecewise linear isomorphism of polyhedral complexes.

  Each face whose covector graph neither has an isolated node in $[d]$ (nor an isolated node in $[n]$) maps to a cell in the covector decomposition of $\tcone(V)$.
\end{theorem}
\begin{proof}
  Ranging over all the faces whose covector graph has no isolated node in $[n]$ we obtain the bijection with Proposition
  \ref{prop:projection_face_sector}.  The definition of the covector of a cell combined with Lemma \ref{lemma:sectors_tropical_hull} characterizes when a cell in $\typedecomp{V}$ is contained in the tropical cone generated by the
  columns of $V$.
\end{proof}
\goodbreak

With Theorem~\ref{thm:regular_subdivison} the former implies the following.

\begin{corollary}[Structure Theorem of Tropical Convexity] \label{coro:dual_sub_tcone}
  The covector decomposition $\typedecomp{V}$ of $\RR^d$ is dual to the regular subdivision of the polytope
  \[
  \conv \SetOf{\pair{e_i}{e_j} \in \RR^d \times \RR^n}{(i,j) \in \Bgraph(V)}
  \]
  with weights given by $V$.  Moreover, the covector decomposition of $\tcone(V)$ is dual to the poset of interior cells.
\end{corollary}

The result above is the same as \cite[Corollary~4.2]{FinkRincon:1305.6329}; their proof is based on mixed subdivisions
and the Cayley Trick \cite[\S9.2]{Triangulations}.

Note that the envelope of a matrix whose coefficients are $0$ or $\infty$ is a braid cone, and so Theorem~\ref{thm:partition} applies to describe the combinatorics.
The min-tropical cones corresponding to these matrices are tropical analogues of ordinary $0/1$-polytopes.

\begin{corollary}
  Let $V$ be a $d\times n$-matrix whose coefficients are $\infty$ or $0$.  A partition $E$ of $[d]\dcup [n]$ defines a face of the polyhedral fan $\typedecomp{V} \subseteq \RR^d$ with apex $\0$ if and only if 
  \begin{enumerate}[label = (\roman*)]
  \item for each part $K$ of $E$ the induced subgraph of $\Bgraph(V)$ on $K$ is weakly connected, 
  \item the minor of $\Bgraph(V)$ which results from simultaneously contracting each part of $E$ does not contain any directed cycle, and
  \item no part of $E$ is a single element of $[n]$.
  \end{enumerate}
\end{corollary}

As projections of the faces of the envelope $\envelope(V)$ the cones in such a fan can encode an arbitrary digraph on $d$ nodes.

\begin{example}
  The maximal cell in Figure \ref{figure:infinite_convex_hull} is the intersection of the sectors $S_3(\transpose{(0,1,0)})$, $S_2(\transpose{(0,1,2)})$ and $S_1(\transpose{(0,\infty,\infty)})$. On the other hand, it is the projection of the face of the envelope $\envelope(V)$ corresponding to the graph on three nodes with the arcs $(1,3),(2,2),(3,1)$ for the matrix $V$ from Example \ref{example:envelope_with_infty}. 

The recession cone of this face is given by the graph in Figure \ref{figure:cone_graph_face_envelope.tex}. It has the strong components $1\times3$ and $23\times12$. Hence, a minimal generator of the pointed part of the cone is $\transpose{(0,1,1;1,1,0)}$ by Proposition \ref{prop:rays_of_cone}. This projects to the ray generated as the positive span of $\transpose{(0,1,1)}$ which is indeed contained in the tropical cone $\tcone(V)$.
\end{example}

\begin{figure}[ht]
  \centering
  \inclpic{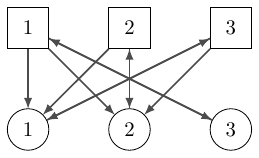}
  \caption{Bipartite graph for the face projecting to the maximal cell in Figure \ref{figure:infinite_convex_hull}}
  \label{figure:cone_graph_face_envelope.tex}
\end{figure}

\begin{remark}
  Clearly, we can also project the envelope $\envelope(V)$ onto the $[n]$ coordinates of the lower shore.  This yields a
  covector decomposition of $\RR^n$ induced by the $d$ rows of the matrix $V$.  Applying Theorem~\ref{thm:envelope} to the
  transpose $\transpose{V}$ gives an isomorphism between the envelope faces without any isolated node in $[d]$ and the
  cells in the covector decomposition of $\RR^n$ induced by the rows of $V$.

  Therefore, the cells whose covector graphs do not have any isolated node in their covector graphs project affinely isomorphic
  to $\RR^d$ as well as to $\RR^n$.  This entails an isomorphism between the covector decompositions of $\tcone(V)$ and
  $\tcone(\transpose{V})$.
\end{remark}

\begin{proposition} \label{prop:char_type_graph}
   Let $G$ be a subgraph of $[d]\times[n]$. Then the following statements are equivalent.
   \begin{enumerate}[label = (\roman*)]
   \item \label{item:point_env}
     There is a point $(y,z) \in \envelope(V)=Q(W)$ for which the inequality corresponding to $(i,j) \in \Gamma(W)$ is attained with equality if and only if $(i,j) \in G$.
   \item \label{item:matching}
     \begin{enumerate}[label = (\alph*)]
     \item For every pair of subsets $D \subseteq [d]$ and $N \subseteq [n]$ with $|D| = |N|$, every perfect matching of $G$ restricted to
       $D \dcup N$ is a minimal matching of the complete bipartite graph $D\times N$ with the weights given by the corresponding submatrix of~$V$;
     \item if there are more minimal perfect matchings in $D\times N$ then each of them is contained in $G$.
     \end{enumerate}
   \item \label{item:cycles}
     \begin{enumerate}[label = (\alph*)]
     \item The graph $\Gamma(W\#G)$ does not have any negative cycle, and
     \item every arc of $\Gamma(W)$ in $\Gamma(W\#G)$ that is contained in a cycle of weight zero is contained in $G$.
     \end{enumerate}
   \end{enumerate}
\end{proposition}
\begin{proof}
  To conclude \ref{item:matching} from \ref{item:point_env} let $D \subseteq [d]$ and $N \subseteq [n]$ with $|D| = |N|$ so that there is a
  perfect matching $\mathcal{M}_0$ in $D\times N \cap G$. Let $\mathcal{M}_1$ be any other perfect matching in
  $D\times N$. Then considering the corresponding inequalities and equations implies after summing up and reordering
  \[
  \sum_{(i,j) \in \mathcal{M}_0} v_{ij} \ = \ \sum_{i \in D} y_i - \sum_{j \in N} z_j \ \leq \ \sum_{(i,j) \in \mathcal{M}_1} v_{ij} \enspace .
  \]
  Therefore, $\mathcal{M}_0$ is a minimal perfect matching. Furthermore, if $\mathcal{M}_1$ is also a minimal perfect
  matching, then equality follows in the former inequality. That implies the equations $y_i - z_j = v_{ij}$ for every
  $(i,j) \in \mathcal{M}_1$. Hence, every arc in $\mathcal{M}_1$ has to be contained in $G$.

  We now want to show that this implies \ref{item:cycles}. For this, we consider a non-positive cycle in $\Gamma(W\#G)$
  with vertex set $D \dcup N$. Let $A_W$ be the set of arcs directed from $[d]$ to $[n]$ and $A_G$ the set of arcs directed
  from $[n]$ to $[d]$. Since $\Gamma(W\#G)$ is bipartite, this implies $|D| = |N| = |A_W| = |A_G|$ and the arc sets 
  $A_W$ and $A_G$ define perfect matchings in $D\times N$.

  By definition of $\Gamma(W\#G)$ we obtain for the weight of the cycle
  \[
  \sum_{(i,j) \in A_W} v_{ij} + \sum_{(j,i) \in A_G} (-v_{ij})\ \leq\ 0 \qquad \text{ or, equivalently, } \qquad \sum_{(i,j) \in A_W} v_{ij}\ \leq\ \sum_{(j,i) \in A_G} v_{ij} \enspace .
  \]
  If the inequality is strict, this contradicts the minimality of the matching via \ref{item:matching}. If the cycle has
  weight zero and the inequality becomes an equality, this implies that $A_W$ also represents a minimal perfect
  matching. With \ref{item:matching} every arc in $A_W$ is also in $G$ then.

  The final goal is to lead \ref{item:cycles} back to \ref{item:point_env}. If $\Gamma(W\#G)$ does not contain a
  negative cycle, the weighted digraph polyhedron $Q(W\#G)$ is not empty. Therefore, there is $(y,z)$ in the interior of the face $Q(W\#G)
  \subseteq \RR^d\times\RR^n$. Let $(i,j)$ be some arc of $\Gamma(W)$. If the equality $y_i - z_j = v_{ij}$ holds,
  Lemma \ref{lemma:zero_cycles}\ref{lemma:item:zero_cycles} yields that there is a cycle of weight zero containing the arc $(i,j)$. With
  \ref{item:cycles} we obtain $(i,j) \in G$. On the other hand, for $(i,j) \in G$, the graph $\Gamma(W\#G)$ contains the
  cycle $(i,j,i)$ of weight zero, and the claim follows.
\end{proof}

  Together with Proposition~\ref{prop:projection_face_sector} this also gives a characterization for the covector graphs which are contained in the tropical cone $\tcone(V)$.
  Furthermore, we obtain a corollary concerning the dimension of a cell.
  
\begin{corollary} \label{coro:dim_type}
   If $G \subseteq \Bgraph(V)$ is a covector graph for $V$, the dimension of $F_G(W)$ and thus of $X_G(V)$ equals the number of weak components of $G$.
\end{corollary}
\begin{proof}
  By property \ref{item:cycles} of Proposition~\ref{prop:char_type_graph} two nodes in $[d] \dcup [n]$ are connected by a
  path in $G$ if and only if they are in a cycle of weight zero in $\Gamma(W\#G)$. By Lemma~\ref{lemma:zero_cycles}\ref{lemma:item:zero_cycles}
  these cycles exactly define the equality partition of $F_G(W)$. Finally, Lemma~\ref{lemma:dim equality graph} connects this to the
  dimension. Furthermore, Proposition~\ref{prop:projection_face_sector} shows the equality for $F_G(W)$ and $X_G(V)$.
\end{proof}

\begin{remark}
  The envelope of $V$ is the set of points $(y,z)$ satisfying
  \[
  y_i - z_j \ \leq \ v_{ij} \quad \text{ for } (i,j) \in \Bgraph \, .
  \]
  Substituting $z_j$ by $-z_j$ yields
  \begin{equation} \label{eq:max_cover_lp}
    y_i + z_j \ \leq \ v_{ij} \quad \text{ for } (i,j) \in \Bgraph \enspace ,
  \end{equation}
  which is the form of the envelope in \cite{DevelinSturmfels:2004}.  Maximizing the coordinate sum over the polyhedron
  defined in~\eqref{eq:max_cover_lp} is dual to finding a minimum weight matching by Egerv\'ary's Theorem \cite[Theorem
  17.1]{Schrijver:CO:A}.  This gives rise to a primal-dual algorithm for computing matchings and vertex covers; the method is
  explained in detail in \cite[Theorem 11.1]{PapadiCO}.  A partial matching of minimal weight in a subgraph can be
  expanded by growing so-called `Hungarian trees', which are shortest path trees in a modified graph.  The partial
  matchings, which encode tight inequalities in the dual description, are collected in the equality subgraphs. By
  Proposition~\ref{prop:char_type_graph} one can deduce that these equality subgraphs are exactly the covector graphs of the
  dual points $(y,z)$.
\end{remark}


\subsection{Tropical half-spaces}
The sectors $S_i(u)$ with $u_i\neq\infty$ from Lemma~\ref{lem:sectors}, which are responsible for the combinatorial properties of $\min$-tropical point configurations, are precisely the (closures of the) complements of the $\max$-tropical hyperplane with apex $u$.
The same combinatorial objects also control systems of tropical linear inequalities.
To see this it is convenient to switch to $\max$ as the tropical addition now.

Let $c \in \TT_{\min}^d$ and let $I$ be a non-empty proper subset of $[d]$, i.e., $I\neq\emptyset$ and $I\neq[d]$. Then the set $\bigcup_{\ell \in I} S_{\ell}(c)$ is a
\emph{$\max$-tropical half-space} with \emph{apex} $c$.  This is exactly the set of points in $\RR^d$ which satisfies
the homogeneous $\max$-tropical linear inequality
\[
\max_{\ell \in [d] \setminus I} (-c_{\ell} + x_{\ell}) \ \leq \ \max_{\ell \in I} (-c_{\ell} + x_{\ell}) \enspace .
\]
Since here we allow for $\infty$ as a coordinate in $c$ this definition is more general than the one in \cite{MJ:2005}.
Notice that $-c$ is an element of $\TT_{\max}^d$ and that the halfspaces are defined over the $\max$-tropical semiring.
Each tropical cone is the intersection of finitely many tropical half-spaces and conversely.
This is proved in \cite[Theorem~1]{GaubertKatz:11}; note that the proof of \cite[Theorem~3.6]{MJ:2005} (which claims the same) is not valid as it rests on \cite[Proposition~3.3]{MJ:2005}, which is false.
In \cite[\S7.6]{Butkovic:10} it is shown that the solution set of any
system of max-tropical linear equalities is finitely generated.  Since $u\leq v$ holds if and only if $\max(u,v)=v$,
i.e., since in the tropical setting studying systems of linear equalities amounts to the same as studying systems of
linear inequalities, that result is essentially equivalent to \cite[Theorem~1]{GaubertKatz:11}.

\begin{remark} \label{rem:wdp_trop_ineq}
  Let $W$ be a $k{\times}k$-matrix.  Each defining inequality \eqref{eq:defining} of the weighted digraph polyhedron $Q(W)$
  can be rewritten as 
  \[
  x_i - w_{ij}\ \leq \ x_j \qquad \text{ for each arc $(i,j)$ in $\Gamma(W)$} \enspace .
  \]
  Fixing $j$ and varying $i$ then yields
  \[
  \max_{i \in [k]}(x_i - w_{ij})  \ \leq \ x_j \qquad \text{ for each } j \in [k] \enspace .
  \]
  Looking at all $j$ simultaneously we obtain the inequality
  \[
  (-\transpose{W}) \odot_{\max} x \ \leq \ x
  \]
  of column vectors.  This means that each weighted digraph polyhedron is a max-tropical cone.  In
  \cite[\S1.6.2 and \S2]{Butkovic:10} a vector $x$ satisfying the inequality above is called a `subeigenvector' of the matrix $-\transpose{W}$.
\end{remark}

We now want to introduce notation for inequality descriptions of tropical cones which is suitable for our combinatorial
approach.  Let $V=(v_{ij}) \in \TT_{\min}^{d{\times}n}$ and let $\Psi$ be a subgraph of the complete bipartite graph
$[d]\times[n]$ with arcs directed from $[d]$ to $[n]$. We define
\begin{equation} \label{eq:section_complex}
  \thalf(V,\Psi) \ = \ \bigcap_{j \in [n]} \bigcup_{(i,j) \in \Psi} S_{i}(v^{(j)}) \enspace .
\end{equation}
That is, $\thalf(V,\Psi)$ comprises those points $x \in \RR^d$ which satisfy the homogeneous $\max$-tropical linear inequalities
\[
\max_{i \in [d],\, (i,j) \not\in \Psi}(-v_{ij} + x_i) \ \leq \ \max_{i \in [d],\, (i,j) \in \Psi}(-v_{ij} + x_i)
\]
for each $j\in[n]$.
In our notation the columns of the matrix $V$ collect the apices of the tropical half-spaces, and the graph $\Psi$ lists the sectors per half-space.
In \cite[\S7]{Butkovic:10} exterior descriptions of tropical cones like
\eqref{eq:section_complex} are discussed under the name `two-sided max-linear systems'.  To phrase our results below it
is convenient to introduce two sets of subgraphs of $[d]\times[n]$, both of which depend on $\Psi$.  We let
\[
\begin{aligned}
  \cG_\Psi \ &= \ \SetOf{G \subseteq \Psi}{\text{every node in } [n] \text{ has degree } 1 \text{ in } G} \qquad \text{and}\\
  \cH_\Psi \ &= \ \SetOf{H \subseteq [d]\times[n]}{\text{every node in } [n] \text{ has degree } \geq 1 \text{ in }
    \Psi\cap H} \enspace ,
\end{aligned}
\]
which gives the following.
\begin{proposition}\label{prop:graphs_types_cells}
  For each graph $H\in\cH_\Psi$ the cell $X_H$, which may be empty, is contained in $\thalf(V,\Psi)$.
  Moreover, $\cG_\Psi \subseteq \cH_\Psi$, and we have
  \[
  \thalf(V,\Psi) 
  \ = \ \bigcup_{G \in \cG_\Psi} \bigcap_{(i,j) \in G} S_{i}(v^{(j)}) \ = \ \bigcup_{H \in \cH_\Psi} \bigcap_{(i,j) \in H}  S_{i}(v^{(j)}) \enspace .
  \]
\end{proposition}
\begin{proof}
  Here the first equality is obtained by reordering the intersections and unions in the Definition
  \eqref{eq:section_complex}.  For the second equality notice that $\cG_\Psi \subseteq \cH_\Psi$. Since for every graph
  $H \in \cH_\Psi$ there is a graph $G \in \cG_\Psi$ so that $X_H(V) \subseteq X_G(V)$ the claim follows.
\end{proof}
The preceding proposition says that a cell $X_G(V)=\bigcap_{(i,j)\in G} S_i(v^{(j)})$ in the covector decomposition
$\typedecomp{V}$ of $\RR^d$ with covector graph $G \subseteq [d]\times[n]$ is contained in the $\max$-tropical cone
$\thalf(V,\Psi)$ if and only if no node in $[n]$ is isolated in the intersection of $G$ and $\Psi$.  Moreover,
$\thalf(V,\Psi)$ is a union of cells. In this way the Proposition~\ref{prop:graphs_types_cells} can be seen as some kind
of a dual version of \cite[Theorem~15]{DevelinSturmfels:2004}, which is a key structural result in tropical convexity.
\begin{corollary}
  The covector decomposition of $\thalf(V,\Psi)$ induced by the columns of $V$ is dual to a subcomplex of the regular
  subdivision of $\Delta_{d-1}\times\Delta_{n-1}$ with weights given by $V$.
\end{corollary}

\begin{example}
  The apices $\transpose{(0,1,1)}$ and $\transpose{(0,2,1)}$ induce the cell decomposition depicted in
  Figure~\ref{figure:not_max_cell}. Every node in the bottom shore in the graph $G$ to the right has degree $1$. Hence,
  it is the kind of graph contained in $\cG_{\Psi}$ for some appropriate $\Psi$ (for example $G$
  itself). However, the corresponding cell is not full-dimensional since the apices are not in general position. Indeed,
  the covector graph of this cell is obtained from $G$ by adding the arcs $(3,1)$ and $(1,2)$.
\end{example}

\begin{figure}[ht]
  \centering
  \inclpic{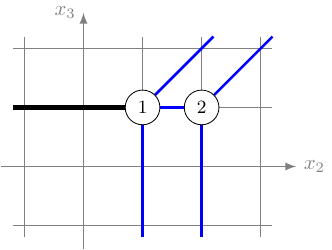}
\hspace*{4em}
  \inclpic{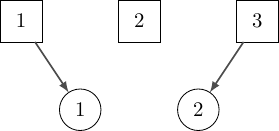}
  \caption{The left figure shows the cell decomposition induced by two apices which are not in general position. The
    right figure depicts a graph $G$ corresponding to the black marked cell $X_G$ on the left}
  \label{figure:not_max_cell}
\end{figure}

\begin{remark}\label{rem:tangent_digraph}
  The \emph{tangent digraph}, defined in \cite[\S3.1]{ABGJ-Simplex:A}, describes the local combinatorics at a cell $C$
  of $\thalf(V,\Psi)$.  This is related to the above as follows. Deleting all nodes in $[n]$ (and incident arcs) for
  which all incident arcs are contained in $\Psi$ in the covector graph $\typegraph{C}$ and forgetting about the orientation
  yields the \emph{tangent graph} $\tangentgraph(C)$ of \cite[\S3.1]{ABGJ-Simplex:A}.  By taking the orientation into
  account and reversing every arc in $\tangentgraph(C)$ which is not in $\tangentgraph(C) \cap \Psi$ from the bottom
  shore $[n]$ (corresponding to the hyperplane apices) to the top shore $[d]$ (corresponding to the coordinate
  directions) we obtain the tangent digraph.
\end{remark}

Proposition \ref{prop:graphs_types_cells} implies that the $\max$-tropical cone $\thalf(V,\Psi)$ is compatible with the
covector decomposition of $\RR^d$ induced by $V$.  Thus it makes sense to talk about the \emph{covector decomposition} of a
$\max$-tropical cone with respect to a fixed system of defining tropical half-spaces.  This is the polyhedral
decomposition formed by the cells which happen to lie in the tropical cone.  A tropical cone is \emph{pure} if each cell
in its covector decomposition which is maximal with respect to inclusion shares the same dimension.  While the covector
decomposition does depend on the choice of the defining inequalities, pureness does not.

The \emph{tropical determinant} of a square matrix $W=(w_{ij})\in\TT_{\min}^{k{\times}k}$ is 
\begin{equation}\label{eq:tdet}
  \begin{aligned}
    \tdet W \ &= \ \bigoplus_{\sigma\in\Sym(k)} \bigodot_{i\in[k]} w_{i,\sigma(i)} \\
    &= \ \min_{\sigma\in\Sym(k)} (w_{1,\sigma(1)}+w_{2,\sigma(2)}+\dots+w_{k,\sigma(k)}) \enspace ,
  \end{aligned}
\end{equation}
which is the same as the solution to a minimum weight bipartite matching problem in the complete bipartite graph
$[k]\times[k]$.  The tropical determinant \emph{vanishes} if the minimum in \eqref{eq:tdet} equals $\infty$ or if it is
attained at least twice. In \cite[\S6.2.1]{Butkovic:10} a square matrix whose tropical determinant does not vanish is
called `strongly regular'. A not necessarily square matrix is \emph{tropically generic} if the tropical determinant of
no square submatrix vanishes.  A finite set of points is in \emph{tropically general position} if 
any matrix whose columns (or rows) represent those points is tropically generic.  Develin and Yu conjectured that
a tropical cone is pure and full-dimensional if and only if it has a half-space description in which the apices of these
half-spaces are in general position \cite[Conjecture~2.11]{DevelinYu:2007}.  The next result confirms one of the two
implications.
\begin{theorem}\label{thm:pure}
  Let $V$ and $\Psi$ be as before. If $V$ is tropically generic with respect to the tropical semiring $\TT_{\min}$ then
  the $\max$-tropical cone $\thalf(V,\Psi)$ is pure and full-dimensional.
\end{theorem}
\begin{proof}
  As in Proposition \ref{prop:graphs_types_cells} we consider the graph class $\cG_\Psi$. If we can show that each
  ordinary polyhedron $X_G(V) = \bigcap_{(i,j) \in G} S_{i}(v^{(j)})$ for $G \in \cG_\Psi$ is either full-dimensional or
  empty then the claim follows.  Proposition \ref{prop:projection_face_sector} implies that $X_G(V)$ is the projection
  of the weighted digraph polyhedron $Q(W\#G)$, which is a face of $\envelope(V)=Q(W)$. Assume that $Q(W\#G)$ is
  feasible.  We have to show that $X_G(V)$ is full-dimensional, i.e., it suffices to show that $\dim Q(W\#G)=d$.

  In view of Proposition~\ref{prop:char_type_graph} together with Corollary~\ref{coro:dim_type} this will follow if we
  can show that no two nodes in $[n]$ are contained in a cycle of weight zero in $\Gamma(W\#G)$. Aiming at an indirect
  argument we suppose that such a cycle exists. Let $D \dcup N$ be the vertex set of the zero cycle
  $(d_1,n_1,d_2,n_2,\dots,d_1)$.  We have $|D| = |N|$. Then the arcs $(d_1,n_1),(d_2,n_2),\dots$ form a perfect matching
  $\cM$ in $D\times N$ whose weight $\sum_i v_{d_i,n_i}$ is minimal by Proposition \ref{prop:char_type_graph}. The
  complementary arcs $(n_1,d_2),(n_2,d_3),\dots$ of the cycle yield a second matching whose weight is the same as the
  weight of $\cM$ since the total weight of the cycle is zero.  This entails that the minimum
  \[
  \min_{\sigma} \sum_{i \in D} v_{i \sigma(i)}  \enspace ,
  \]
  where $\sigma$ ranges over all bijections from $D$ to $N$, is attained at least twice for the submatrix of $V$ indexed
  by $D\times N$.  Hence, the apices are not in general position, and this is the desired contradiction.
\end{proof}
Since the matrix $V$ is tropically generic it is immediate that $\tcone(V)$ has at least one full-dimensional cell;
e.g., see \cite[Theorem 6.2.18]{Butkovic:10} or \cite[Proposition~24]{DevelinSturmfels:2004}.  Yet, in general
$\tcone(V)$ is not pure; see Example~\ref{ex:type_decomposition}.
The following shows that the reverse direction of Theorem~\ref{thm:pure} does not hold.
\begin{example}\label{exmp:pure}
  For
  \[
  V \ = \ \begin{pmatrix} 0 & 0 & 0 & 0 & 0 \\ 3 & 2 & 1 & \infty & \infty \\ 2 & 2 & \infty & 1 & 3 \end{pmatrix}
  \]
  and $\Psi$ as in Figure~\ref{figure:feasible_graph} we are interested in the $\max$-tropical cone $C=\thalf(V,\Psi)$.
  Now $C$ is pure, but the first two columns, $\transpose{(0,3,2)}$ and $\transpose{(0,2,2)}$, of the matrix $V$ are not
  in general position with respect to $\min$.  Notice that each one of the apices of the three remaining tropical
  half-spaces can be moved without changing the feasible set $C$.  However, the first two tropical half-spaces are
  \emph{essential} in the sense that they occur in any exterior description of~$C$.
\end{example}

\begin{figure}[ht]
  \centering
  \inclpic{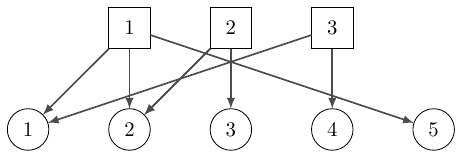}
  \caption{The graph $\Psi$ for the $\max$-tropical cone $C=\thalf(V,\Psi)$ from Example~\ref{exmp:pure} and Figure~\ref{figure:pure_counterexample}}
  \label{figure:feasible_graph}
\end{figure}

\begin{figure}[ht]
  \centering
  \inclpic{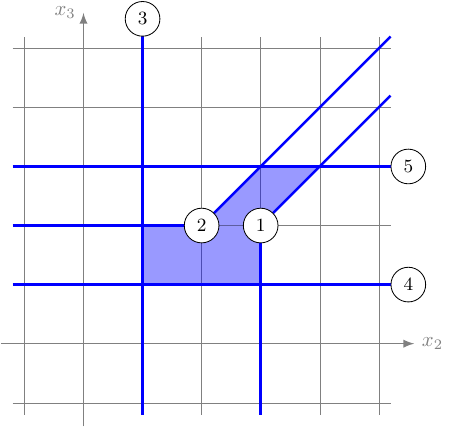}
  \caption{The pure $\max$-tropical cone $C$ from Example~\ref{exmp:pure}.  The apices of any $\max$-tropical half-space
    description are not in general position with respect to $\min$}
  \label{figure:pure_counterexample}
\end{figure}

A related conjecture from the same paper \cite[Conjecture~2.10]{DevelinYu:2007} was recently resolved by Allamigeon and
Katz \cite{AllamigeonKatz:1408.6176}.

\subsection{Polytropes} \label{sec:polytropes}

A \emph{polytrope} is a tropical cone $P=\tcone(V)$ for $V\in\RR^{d\times n}$, i.e., with a generating matrix with
finite coefficients, which is also convex in the ordinary sense.  In that case $d$ generators suffice
\cite[Proposition~18]{DevelinSturmfels:2004} and \cite[Theorem~7]{JoswigKulas:2010}. Therefore we may assume that
$n=d$. From this we obtain $\tcone(V)=Q(V)=Q(V^*)$ in view of Remark~\ref{rem:wdp_min_cone}, and thus any polytrope is a
weighted digraph polyhedron; see also \cite[Proposition~10]{JoswigKulas:2010}.
Yet another argument for the same goes through Theorem~\ref{thm:envelope} and Lemma~\ref{lemma: face of wdp}.
This is slightly more general as it takes $\infty$ coefficients into account.
Moreover, the covector decomposition of $P$ induced by the square matrix $V$ has a single cell.
Its projection to the tropical projective torus $\RR^d/\RR\1$ is bounded, namely the polytrope $P$ itself.
The latter also gives a max-tropical exterior description.
The polytropes are exactly the `alcoved polytopes of type A' of Lam and Postnikov \cite{LamPostnikov:2007}.
The weighted digraph polyhedra form the natural generalization to polyhedra which are not necessarily bounded.
We sum up our discussion in the following statement.

\begin{proposition}
  Let $V\in\TT_{\min}^{d\times n}$ such that the $\min$-tropical cone $\tcone(V)$ is also convex in the ordinary
  sense. Then there is a $d{\times}d$-matrix $U$ such that $\tcone(V)=Q(U)$ is a weighted digraph polyhedron.
\end{proposition}

In the context of proving a hardness result on the vertex-enumeration of polyhedra given in terms of inequalities
Khachiyan and al.\ \cite{Khachiyan:2008} study the \emph{circulation polytope} of the digraph~$\Gamma$, which is the
set of all points $u\in\RR^{\Gamma}$ satisfying
\begin{align*}
  \sum_{j: (i,j) \in \Gamma}u_{ij} - \sum_{\ell: (\ell, i) \in \Gamma}u_{\ell i} \ &= \ 0 \quad  \text{ for all } i \in [k] \\
  \sum_{(i,j) \in \Gamma} u_{ij} \ &= \ 1  \\
  0 \ &\leq \ u_{ij} \quad \text{ for all } (i,j) \in \Gamma \enspace .
\end{align*}
The support set $\smallSetOf{(i,j) \in \Gamma}{u_{ij} \neq 0}$ of a vertex of the circulation polytope defines a cycle in $\Gamma$.  Hence, by
Lemma~\ref{lemma:feasible}, minimizing the weight function $\gamma(W)$ over the circulation polytope yields a
certificate for the feasibility of $Q(W)$.  Tran uses this approach to characterize the feasibility of polytropes in
terms of ordinary inequalities \cite[\S3]{Tran:2013}.

\subsection{Covector decompositions of tropical projective spaces} \label{sec:projective}

The \emph{tropical projective space} $\TPmin^{d-1}$ is defined as the quotient of
$\TT_{\min}^d\setminus\{\transpose{(\infty,\infty,\dots,\infty)}\}$ modulo $\RR\1$.  That is, its points are equivalence
classes of vectors with coefficients in $\TT_{\min}=\RR\cup\{\infty\}$ with at least one finite entry, up to differences
by a real constant; see \cite[Example~3.10]{Mikhalkin:2006}.  The tropical projective space $\TPmin^{d-1}$ is a
natural compactification of the tropical projective torus $\RR^d/\RR\1$.
It is easy to see that the pair $(\TPmin^{d-1},\RR^d/\RR\1)$ is homeomorphic to the pair of a $(d{-}1)$-simplex and its interior.

We assume that $V\in\TT_{\min}^{d\times n}$ has no column identically $\infty$. Then $V$ gives rise to a configuration of $n$
labeled points in $\TPmin^{d-1}$.  The covector decomposition $\typedecomp{V}$ of $\RR^d$ does not change if we add a
real constant to the entries in any column.  So it is an invariant of that point configuration, and, moreover,
$\typedecomp{V}$ induces a covector decomposition of the tropical projective torus $\RR^d/\RR\1$.  Yet it makes sense to
study tropical convexity and tropical cones also within the compactification $\TPmin^{d-1}$.   Our goal is to describe a
decomposition of the tropical projective space into cells.  Let $Z$ be a proper subset of $[d]$.  We consider the matrix
obtained by removing from $V$ all columns $j$ for which there is an $i \in Z$ with
$v_{ij}\not=\infty$.  Each row of the resulting matrix with a label in $Z$ has only $\infty$ as coefficients.
Removing these rows yields yet another matrix, which we denote as $V(Z)$.  Now this matrix induces a covector
decomposition of the \emph{boundary stratum}
\[
\TPmin^{d-1}(Z) \ = \ \SetOf{(p_1,p_2,\dots,p_d)\in\TPmin^{d-1}}{p_i=\infty \text{ if and only if } i\in Z} \enspace,
\]
which is a copy of the tropical projective torus of dimension $d-1-|Z|$.  In particular, we have
$\TPmin^{d-1}(\emptyset)=\RR^d/\RR\1$.  Notice that for the induced covector decomposition we keep the original labels
of the columns and the rows. 

For $K \subseteq [d]$ let $b^{(K)}$ be the vector in $\TT_{\min}^d$ with  
\[
b_i^{(K)} \ = \ 
\begin{cases}
  0 & \text{ for } i \in [d] \setminus K \\
  \infty & \text{ for } i \in K 
\end{cases} \enspace .
\]

Consider $u \in \TT_{\min}^d$ and let $\supp(u)=\smallSetOf{i\in[d]}{u_i\neq\infty}$ be the support of $u$. Then the recession cone of the weighted digraph polyhedron $S_i(u)$ is given by the graph on $[d]$ where the nodes in $[d] \setminus \supp(u)$ are isolated and there are arcs from the nodes in $\supp(u) \setminus \{ i \}$ to~$i$, see Equation~\eqref{eq:sector}.
The supports of the rays of $S_i(u)$ are given by the sets in 
\[
\cK \ = \ \SetOf{A \cup B}{A \in \cA, B \in \cB} 
\]
where
\[
\cA = \emptyset \cup \SetOf{M \cup \{i\}}{M \subseteq \supp(u)} \quad \text{and} \quad \cB = \bigl\{M \subseteq [d]\setminus \supp(u)\bigr\} \enspace .
\]
Here, the sets in $\cA$ correspond to the faces of the pointed part of the recession cone of $S_i(u)$ described by Theorem~\ref{thm:partition}.
The set $\cB$ encodes rays arising from the lineality space of $S_i(u)$ which was characterized in Lemma~\ref{lemma:recession_cone}.

So, it is natural to define
\[
 \csec{i}{u} \ = \ S_i(u) \cup \bigcup_{K \in \cK} \left(b^{(K)} + S_i(u)\right) \enspace
\]
where the `$+$'-operator denotes elementwise ordinary addition of $b^{(K)}$ and the set $S_i(u)$.

In the following we will frequently identify subsets of $(\RR\cup\{\infty\})^d$ with their images modulo $\RR/\1$.  In particular, we will typically view $\csec{i}{u}$ with $u_i\neq\infty$ as a subset of $\TPmin^{d-1}$.

\begin{lemma}\label{lem:sectors_boundary}
  The set $\csec{i}{u}$ for $u_i\neq\infty$ is the compactification of the sector $S_i(u)$ in $\TPmin^{d-1}$.
\end{lemma}

Consider a cell $X_G$ in $\typedecomp{V}$ which contains a ray with support $Z$. Let $M$ be the index set of the columns of $V$ with $v_{ij} = \infty$ for all $i \in Z$ and $j \in M$. Construct the submatrix $Y$ of $V$ indexed by $([d] \setminus Z) \times M$ and the graph $H$ as the restriction of $G$ to the node set $([d] \setminus Z) \dcup M$. 
\begin{lemma}\label{lemma:cell_covector_boundary}
The cell decomposition of $\TPmin(Z)^{d-1}$ induced by $Y$ contains the cell $X_G(V) + b^{(Z)}$ which is given by the covector graph $H$ in the decomposition of $\RR^{([d] \setminus Z)}$ by $Y$.
Furthermore, we obtain the alternative description
\[
  b^{(Z)} + X_G(V) \ = \ b^{(Z)} + \bigcap_{(i,j)\in G} \csec{i}{v^{(j)}} \enspace .
\]
\end{lemma}
\begin{proof}
The second claim is merely a reformulation with the definition of $\csec{i}{u}$.

The first claim follows if we show that
\begin{equation}\label{eq:proj_boundary}
  \pi_Z(X_G(V)) \ = \ X_{H}(Y)
\end{equation}
where $\pi_Z$ is the projection onto the coordinates in $[d] \setminus Z$.
 Since any ray is generated by the minimal generators of the pointed part of the recession cone and the generators of the lineality space, at first we assume that $Z$ is the support of a minimal generator of the pointed part of the recession cone. Setting $D'' = Z$ in Corollary~\ref{coro:inf_env} yields that every shortest path is already defined on $([d]\setminus Z) \dcup M$. Furthermore, the support of a generator of the lineality space is given by a weak component by Lemma~\ref{lemma:recession_cone} what implies the same statement about the shortest paths for those generators. 

Summarizing, equation~\eqref{eq:proj_boundary} follows with Lemma~\ref{lemma:projection_wdp}.
\end{proof}

\begin{theorem}\label{thm:covector_decomposition}
  The union of the covector decompositions induced by the matrices $V(Z)$ where $Z$ ranges over all proper
  subsets of $[d]$ yields a piecewise linear decomposition of $\TPmin^{d-1}$.
\end{theorem}
If the graph $\Bgraph(V)$ is weakly connected, then by Lemma~\ref{lemma:recession_cone} the intersection poset generated by the sets $\smallSetOf{\csec{i}{u}}{u_i\neq\infty}$ contains a $0$-dimensional cell, whence that piecewise linear decomposition of $\TPmin^{d-1}$ is a cell complex.
\begin{proof}
  By definition as the common refinement of polyhedral complexes the covector decomposition of $\RR^d/\RR\1$ induced by
  $V$ is a polyhedral complex.  The bounded cells are polytopes and therefore homeomorphic to closed balls.  We need to
  check that the topology works out right for those cells which are unbounded in $\RR^d/\RR\1$.  This is gotten from an
  induction on~$d$ as follows.  In the base case $d=1$ there is nothing to show since the tropical projective torus
  $\RR^1/\RR\1$ is a single point.  For $d\geq 2$, by induction, we may assume that the covector decomposition induced
  on the closure
  \[
  \SetOf{(p_1,p_2,\dots,p_d)\in\TPmin^{d-1}}{p_i=\infty \text{ if } i\in Z} \enspace,
  \]
  of $\TPmin^{d-1}(Z)$ yields a cell decomposition if $Z$ is not empty.  Now consider $Z = \emptyset$ and let $X_G(V)$ be an unbounded cell with covector $G=(G_1,G_2,\dots,G_d)$.
By Lemma~\ref{lemma:cell_covector_boundary}, the closure of $X_G(V)$ in $\TPmin^{d-1}$ is the union of $X_G(V)$ with all the
  cells $X_H(V(Z'))$ where $Z'$ ranges over the supports of the rays contained in $X_G(V)$.
  Here $H$ is the covector which $G$ induces on $\TPmin^{d-1}(Z')$ by omitting those $G_i$ with $i \in Z'$; this union is homeomorphic with a ball.
 The same argument also shows that intersections of cells are unions of cells.
\end{proof}
By construction one can apply Lemma~\ref{lemma:sectors_tropical_hull} also to the cells in the boundary of the tropical projective space to check for containment in $\tcone(V)$. Consider $z \in \TT^d$ and let $\supp(z) = \smallSetOf{i \in [d]}{z_i \neq \infty}$ be its support.

\begin{corollary}\label{coro:sectors_projective_hull}
  The point $z$ is contained in $\tcone(V)$ if and only if for every $i \in \supp(z)$ there is an index $s \in [n]$ with $z \in \csec{i}{v^{(s)}}$ and $\supp(v^{(s)})\subseteq\supp(z)$.
  A point $z \in \TT^d$ is contained in $\tcone(V)$ if and only if for every $i \in [d]$ there is an index $s \in [n]$ with $z \in \csec{i}{v^{(s)}}$.
\end{corollary}

\begin{example}
  Let
  \[
  V' \ = \ \begin{pmatrix} 0 & 0 & 0 & 0 & \infty & \infty \\ 1 & 0 & \infty & \infty & 0 & \infty \\ 2 & -1 &
    \infty & \infty & \infty & 0 \end{pmatrix} \enspace ,
  \]
  where $d=3$ and $n=6$.
  The third and fourth columns of $V'$ are the same.
  Notice that the first three columns correspond to the matrix $V$ from Example~\ref{ex:type_decomposition}.
  With $Z=\{1\}$ we obtain the matrix
  \[\BAtablenotesfalse
  \begin{blockarray}{*3{>{\scriptstyle}{c}}}
    \begin{block}{\Left{$V'(Z)=$}{(}cc)>{\scriptstyle}c}
      0 & \infty & 2 \\
      \infty & 0 & 3 \\
    \end{block}
    5 & 6 &
  \end{blockarray} \enspace ,
  \]
  where we keep the original row and column labels.  
  The one-dimensional tropical projective torus $B=\TPmin^2(\{1\})$ is trivially subdivided; its covector reads $(\bullet,5,6)$. To denote cells in the boundary we use the symbol $\bullet$ at the component corresponding to an apex to mark if the cell is in a common boundary stratum with this apex.
  The union of the $1$-dimensional ball $B$ and the unbounded cell in $\RR^3/\RR\1$ with covector $(1234,-,-)$ yields the $2$-dimensional cell with covector $(1234,5,6)$ in the covector decomposition of $\TPmin^2$ induced by $V'$; see Figure~\ref{fig:signed_cells} and compare with Figure~\ref{figure:tropical_types}.
\end{example}

Notice that, while the tropical projective torus works for min and max alike, the definition of the tropical projective space does
depend on the choice of the tropical addition.

\begin{figure}[pth]
  \centering
  \resizebox{0.6\textwidth}{!}{\inclpic{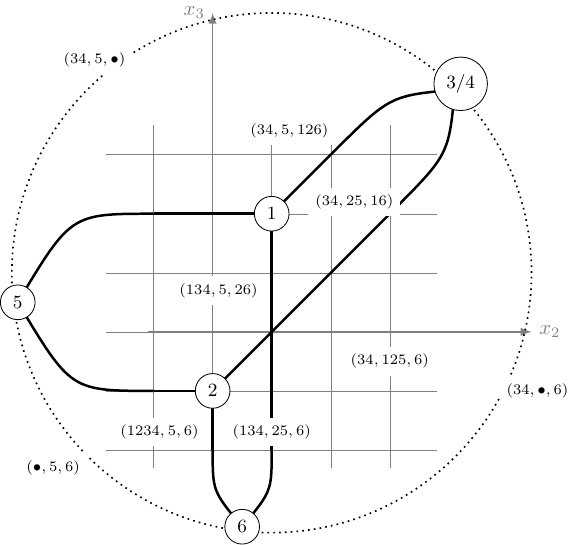}}

  \vspace*{1cm}

  \resizebox{0.6\textwidth}{!}{\inclpic{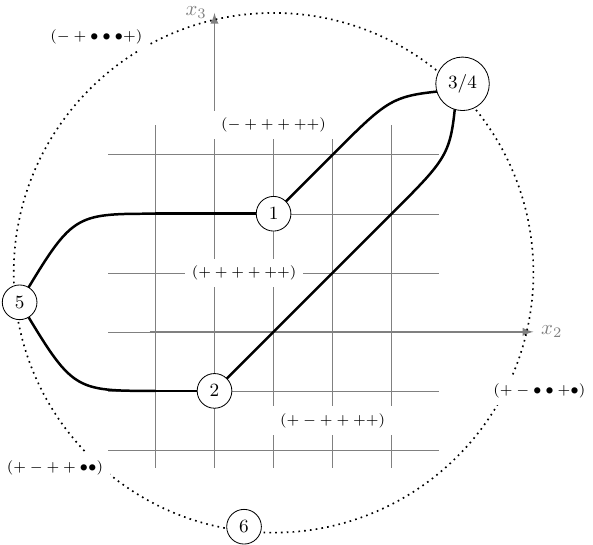}}
  \caption{Covector decomposition (top) and signed cell decomposition (bottom) in the tropical projective plane}
  \label{fig:signed_cells}
\end{figure}

\subsection{Arrangements of  tropical halfspaces}

So far we associated with a matrix $V\in\TT_{\min}^{d\times n}$ the covector decompositions of $\RR^d$ and
$\TPmin^{d-1}$, respectively, and Theorem~\ref{thm:envelope} describes the min-tropical cone $\tcone(V)$ as a union
of their cells.  Choose a subgraph $\Psi$ of the complete bipartite graph $[d]\times[n]$ (with arcs directed from $[d]$
to $[n]$) as in \eqref{eq:section_complex}. This gives rise to the max-tropical cone $\thalf(V,\Psi) = \bigcap_{j \in
  [n]} \bigcup_{(i,j) \in \Psi} S_{i}(v^{(j)})$, which again is a union of cells from the same covector decomposition.  Here
we want to describe yet another cell decomposition of $\RR^d$ (or $\TPmin^{d-1}$), which was introduced in
\cite[\S3.2]{ABGJ-Simplex:A}.

For this, we introduce the max-tropical cone with boundary 
\[
\olthalf(V,\Psi) \ = \ \bigcap_{j \in [n]} \bigcup_{(i,j) \in \Psi} \csec{i}{v^{(j)}} \enspace .
\]

For a vector $\epsilon \in \{\pm\}^n$ of $n$ signs we consider the directed bipartite graph
\[
\signedcell{\Psi}{\epsilon} \ = \ \SetOf{(i,j)\in[d]\times[n]}{\bigl((i,j)\in\Psi \text{ and } \epsilon_j=+\bigr) \text{ or }
  \bigl((i,j)\not\in\Psi \text{ and } \epsilon_j=-\bigr)} \enspace .
\]

The construction of $\signedcell{\Psi}{\epsilon}$ from $\Psi$ amounts to taking the complementary arcs incident to each
node $j \in [n]$ with $\epsilon_j = -$.  We call the max-tropical cone $\olthalf(V,\signedcell{\Psi}{\epsilon})$ the
\emph{inversion} of $\olthalf(V,\Psi)$ with respect to $\epsilon$.  As a subset of $\TPmin^{d-1}$ the inversion may be empty
or not.  In the latter case $\olthalf(V,\Psi_\epsilon)$ is the \emph{signed cell} with respect to $V$, $\Psi$ and
$\epsilon$.  Each generic point, i.e., a point which does not lie on any of the max-tropical hyperplanes whose apices
are columns of $V$, is contained in a unique signed cell.  The trivial inversion with respect to
$\epsilon={+}{+}\cdots{+}$ is the tropical cone $\olthalf(V,\Psi)$ itself.  
Each signed cell is a union of cells of the
covector decomposition.  So Theorem~\ref{thm:covector_decomposition} together with Proposition~\ref{prop:graphs_types_cells} entails the following.

\begin{corollary}\label{cor:signed_cell_decomposition}
  The signed cells $\olthalf(V,\signedcell{\Psi}{\epsilon}) / \RR\1$, where $\epsilon$ ranges over all choices of sign vectors,
  generate a piecewise linear decomposition of $\TPmin^{d-1}$.

  Furthermore, a cell with graph $G$ in the covector decomposition of $\TPmin^{d-1}$ by $V$ is contained in a cell $\olthalf(V,\signedcell{\Psi}{\epsilon})$ if and only if $\Psi_{\epsilon} \cap G$ has no isolated node.
\end{corollary}

The decomposition into signed cells is a tropical analogue of the decomposition into polyhedral cells defined by an
ordinary affine hyperplane arrangement.
As in Theorem~\ref{thm:covector_decomposition} that piecewise linear decomposition is a cell complex, provided that $\Bgraph(V)$ is weakly connected.

\begin{example}
  Figure~\ref{fig:signed_cells} shows the signed cell decomposition of $\TPmin^2$ induced by the matrix $V$ from
  Example~\ref{ex:type_decomposition} with the extra columns $\transpose{(\infty,0,\infty)}$ and $\transpose{(\infty,\infty,0)}$ and the directed bipartite graph $\Psi\subset\{1,2,3\}\times\{1,2,3,4,5\}$ with the six directed
  edges $(1,1),(2,1),(3,2),(1,3),(2,4),(3,5)$. 
The six signed cells correspond to the sign vectors ${+}{+}{+}{+}{+}{+}$, ${-}{+}{+}{+}{+}{+}$ and
  ${+}{-}{+}{+}{+}{+}$.  The remaining $29$ inversions are empty. Finally, the three inversions ${-}{+}{\bullet}{\bullet}{\bullet}{+}$, ${+}{-}{\bullet}{\bullet}{+}{\bullet}$ and ${+}{-}{+}{+}{\bullet}{\bullet}$ form a decomposition of the boundary of the tropical projective plane. 
\end{example}

\section{Concluding remarks}
\noindent
Tropical point configurations, or rather the dual tropical hyperplane arrangements, were generalized to `tropical
oriented matroids' by Ardila and Develin \cite{ArdilaDevelin:2009}.  Horn showed that the latter are equivalent to
subdivisions of a product of simplices which are not necessarily regular~\cite{Horn:2012}.  The tangent digraph discussed
in Remark~\ref{rem:tangent_digraph} also makes sense in the tropical oriented matroid setting.  That graph is the
crucial combinatorial device for the pivoting operation in the tropical simplex algorithm~\cite{ABGJ-Simplex:A}.
\begin{problem}
  Give an oriented matroid version of the tropical simplex algorithm.
\end{problem}

It is worth noting that the axioms for tropical oriented matroids given in \cite{ArdilaDevelin:2009} generalize the
combinatorics of tropical convexity with finite coordinates only.
\begin{problem}
  Generalize the axioms of tropical oriented matroids to cover point configurations or hyperplane arrangements in the
  tropical projective space.
\end{problem}
In view of Theorem~\ref{thm:envelope} and the results in~\cite{Horn:2012} this might be related not necessarily regular
subdivisions of subpolytopes of products of simplices.

\begin{problem}
  How are the signed cell decompositions related to tropical oriented matroids?
\end{problem}

\section*{Acknowledgment}
We are indebted to Xavier Allamigeon, Federico Ardila, Peter Butkovi\v{c}, Veit Wiechert and two anonymous referees for several valuable hints.

\goodbreak

\bibliographystyle{amsplain}
\bibliography{main}

\end{document}